\documentclass[11pt, a4paper]{article}
\usepackage{natbib}
\usepackage{graphicx}
\usepackage{subcaption}
\usepackage{amsmath}
\allowdisplaybreaks[2]
\usepackage[colorlinks=true,linkcolor=blue, citecolor=black]{hyperref}
\usepackage{cleveref}
\usepackage{amssymb}
\usepackage{amsthm}
\usepackage{color}
\usepackage{algorithm}
\usepackage{algorithmic}
\usepackage{xcolor}
\newtheorem{theorem}{Theorem}[section]
\newtheorem{prop}{Proposition}[section]

\newtheorem{lemma}{Lemma}[section]
\numberwithin{equation}{section}

\allowdisplaybreaks[4]

\title{A direct sampling method for magnetic induction tomography}
\author{Junqing Chen\thanks{\footnotesize
Department of Mathematical Sciences, Tsinghua University, Beijing
100084, China. The work of this author was partially supported by the National Natural Science Foundation of China under the grant 92370125. (jqchen@tsinghua.edu.cn).}
\and Chengzhe Jiang\thanks{\footnotesize
Department of Mathematical Sciences, Tsinghua University, Beijing
100084, China. (jiangcz24@mails.tsinghua.edu.cn)}}

\date{} 

\begin{document}

\maketitle

\begin{abstract}
This paper proposes a direct sampling method for the inverse problem of magnetic induction tomography (MIT). Our approach defines a class of point spread functions with explicit expressions, which are computed via inner products, leading to a simple and fast imaging process. We then prove that these point spread functions decay with distance, establishing the theoretical basis of the algorithm. Specific expressions for special cases are also derived to visually demonstrate their attenuation pattern. Numerical experimental results further confirm the efficiency and accuracy of the proposed algorithm.
\end{abstract}
{{\bf Mathematics Subject Classification}(MSC2020): 78A46, 65N21, 35R30}\\
{{\bf Keywords:} magnetic induction tomography, eddy current, direct sampling method }

\section{Introduction}

Magnetic induction tomography (MIT), which is also known as electromagnetic induction tomography or eddy current tomography in different fields, is a promising, non-invasive imaging modality that reconstructs an object’s internal conductivity distribution through electromagnetic induction. Its key advantages, such as non-contact operation and the absence of electrodes, make it particularly suitable for monitoring the pathological changes in biological tissues \cite{MIT1, MIT3, MIT4, Merwa_2005} and nondestructive testing in industrial tomography \cite{PEYTON202277}. The underlying physical process involves drive coils generating an oscillating primary magnetic field, which induces eddy currents within the conductive regions. These eddy currents, in turn, produce a secondary magnetic field that contains information about the conductivity distribution. The mathematical model for MIT is fundamentally based on the eddy current model. In the low-frequency range, for instance below 100MHz,  the displacement current term in Maxwell’s equations can be neglected \cite{Eddy2010}, in which case the eddy current equation serves as a robust and efficient approximation of the full Maxwell’s equations.

The inverse problem of MIT is ill-posed since there is usually no uniqueness and stability of the solution. Currently, the main way to solve this problem is iterative methods, in which iterations are employed to minimize the $L^2$-misfit between the reconstruction outcome and the experiment data. They also use a regularization term to obtain the appropriate regularity of the conductivity distribution. Various iterative optimization methods have been applied, such as the quasi-Newton method, the split Bregman method, etc., and the regularization terms can be Tikhonov regularization, Laplace regularization, and TV regularization \cite{Chen2024, MIT_R2, Wolff_2012, Li_2017, Merwa_2005}. These methods usually require multiple iterations and multiple solutions for the forward problem, resulting in significant time consumption. 

To overcome the huge computational burden of iterative methods for inverse problems, researchers have recently proposed various direct sampling methods, including the linear sampling method \cite{LSM3}, the factorization method \cite{Fac1}, and the reverse time migration method \cite{RTM1, RTM2}, to name a few. Direct sampling methods do not require prior knowledge of the inclusions and have solid theoretical foundations and reliable numerical outcomes. Usually a set of indicators is proposed, taking large values within the target area and small values outside. These direct methods have been successfully applied to other inverse problems, including mainly electromagnetic scattering problems \cite{Scatter1, Scatter2}, acoustic wave problems \cite{RTM1}, and elastic wave problems \cite{RTM3}. Even for the strongly ill-posed electrical impedance tomography (EIT) problem, a specifically designed direct method has been proposed in \cite{Chow2014}. However, the results of direct sampling methods for solving the MIT problem are very limited. For small-volume conductive bodies, an asymptotic expansion of the magnetic field perturbation in the eddy current model has been derived in \cite{MUSIC2}, and based on the expansion, a MUSIC-type algorithm is also constructed to locate the position of the conductors. In addition, a linear sampling method has been proposed for the axisymmetric eddy current model in \cite{LSM1}. The article measures the electric field data within a linear region along the axis, and defines an index function at each point within the sampling area to indicate the position of the conductor. These algorithms are only applicable to specific cases of the eddy current model, and currently, there is no direct sampling method capable of addressing the general cases of MIT. Therefore, it is necessary to propose a new and more universal direct sampling method for the MIT problem.

In this paper, we introduce a direct sampling method for the MIT inverse problem, which is based on a newly designed class of point spread functions. We theoretically demonstrate that these functions decay to zero in regions distant from conductive inclusions, and derive their explicit estimates for several specific cases (see propositions \ref{prop:decay} and \ref{prop:decay2} in Section \ref{sect3}). A key advantage of the proposed method is that the final index function $I$ is defined solely through the inner products, making the imaging process straightforward and computationally efficient. Furthermore, most of the computational workload can be performed offline, which substantially reduces the online reconstruction time and enables rapid imaging.

Our paper is organized as follows. Section \ref{sect2} formulates the MIT problem and its corresponding eddy current model. It also introduces the integral representation formula for the scattered magnetic field, which forms the basis of our direct sampling method. In Section \ref{sect3}, we define the duality product on the measurement surface $\Gamma$ and propose a novel class of point spread functions, followed by a proof of their attenuation property. The complete direct sampling algorithm is then presented. In Section \ref{sect4} presents some numerical examples to validate the effectiveness of our proposed method. Finally, we present some concluding remarks of this study in Section \ref{sect5}.

\section{Eddy current model}\label{sect2}
In this section, we first summarize the eddy current model for magnetic induction tomography and then present the basic setup for the corresponding inverse problem.
Let $\mu$ be the magnetic permeability which is assumed to be constant in $\mathbb{R}^3$. Let $\sigma\in L^\infty (\mathbb R^3)$ be the electric conductivity with compact support $D\subset\mathbb{R}^3$, where $D$ has Lipschitz boundary. We further assume that in $D$ the conductivity $\sigma$ is uniformly positive, i.e., $\sigma\ge\sigma_0>0$ almost everywhere in $D$. Then the eddy current fields $(\boldsymbol E,\boldsymbol H)$ satisfy the following equations:
\begin{equation}\label{eq:eddy-E&H}
    \begin{cases}
    \nabla\times\boldsymbol E=i\omega\mu\boldsymbol H & \text{in}\,\,\mathbb{R}^3, \\
    \nabla\times\boldsymbol H=\sigma\boldsymbol E + \boldsymbol J_0 & \text{in}\,\,\mathbb{R}^3, \\
    \boldsymbol E(\boldsymbol{x})=O(|\boldsymbol{x}|^{-1}),\,
    \boldsymbol H(\boldsymbol{x})=O(|\boldsymbol{x}|^{-1}) & \text{as}\,\,|\boldsymbol x|\to\infty.
    \end{cases}
\end{equation}
Here $\boldsymbol{J}_0$ is the source current located outside $D$. Moreover, we assume $\boldsymbol{J}_0$ is divergence free and has compact support. By eliminating $\boldsymbol H$ in \eqref{eq:eddy-E&H}, we obtain
\begin{equation}\label{eddy-E}
    \begin{cases}
    \nabla\times(\mu^{-1}\nabla\times\boldsymbol{E})-i\omega\sigma\boldsymbol{E}=i\omega\boldsymbol{J}_0 &\text{in}\,\,\mathbb{R}^3,\\
    \nabla\cdot\boldsymbol{E}=0 &\text{in}\,\,\mathbb{R}^3\setminus D,\\
    \boldsymbol E(\boldsymbol{x})=O(|\boldsymbol{x}|^{-1}) &\text{as}\,\,|\boldsymbol x|\to\infty.
    \end{cases}
\end{equation}
The uniqueness and existence of the solution of the problem \eqref{eddy-E} is known \cite{Uniqueness} when $\boldsymbol E$ is constrained to space
\begin{equation*}\label{eq:space-E}
    \begin{aligned}
    \boldsymbol X(\mathbb R^3):=\left\{
    \boldsymbol u\left|
    \frac{\boldsymbol u}{\sqrt{1+|\boldsymbol x|^2}}\in \boldsymbol L^2(\mathbb R^3),\,
    \nabla\times\boldsymbol u\in\boldsymbol L^2(\mathbb R^3),\,
    \right.\right.\qquad\qquad\\
    \left.
    \nabla\cdot\boldsymbol u=0\,\,\,\text{in}\,\,\,\mathbb R^3\setminus D,\,
    \int_{\partial D_\alpha}\boldsymbol u|_{\mathbb R^3\setminus D}\cdot\boldsymbol n=0
    \right\}.
    \end{aligned}
\end{equation*}
Here $\partial D_\alpha$ stands for the connected components of $\partial D$, and $\boldsymbol n$ denotes the unit outward normal on $\partial D$. The constraint $\int_{\partial D_\alpha}\boldsymbol u|_{\mathbb R^3\setminus D}\cdot\boldsymbol n=0$ only serves to enforce the uniqueness of $\boldsymbol E$, which is not essential \cite{Uniqueness}. 

We denote $\boldsymbol E_0$ the solution of homogeneous model with $\sigma=0$ in $\mathbb R^3$, which solves
\begin{equation}\label{eddy-E0}
    \begin{cases}
    \nabla\times(\mu^{-1}\nabla\times\boldsymbol{E}_0)=i\omega\boldsymbol{J}_0 &\text{in}\,\,\mathbb{R}^3,\\
    \nabla\cdot\boldsymbol{E}_0=0 &\text{in}\,\,\mathbb{R}^3,\\
    \boldsymbol E_0(\boldsymbol{x})=O(|\boldsymbol{x}|^{-1}) &\text{as}\,\,|\boldsymbol x|\to\infty.
    \end{cases}
\end{equation}
Similarly, problem \eqref{eddy-E0} has a unique solution when $\boldsymbol E_0$ is constrained to space
\begin{equation*}\label{eq:space-E0}
    \begin{aligned}
    \boldsymbol X_0(\mathbb R^3):=\left\{
    \boldsymbol u\left|
    \frac{\boldsymbol u}{\sqrt{1+|\boldsymbol x|^2}}\in \boldsymbol L^2(\mathbb R^3),\,
    \nabla\times\boldsymbol u\in\boldsymbol L^2(\mathbb R^3),~\nabla\cdot\boldsymbol{u}=0~\text{in}~\mathbb R^3\right.\right\}.
    \end{aligned}
\end{equation*}

Now we present the integral representation formula for the magnetic perturbation field arising from the conductive inclusion \cite{AMMARI2014}. This formula serves as the foundation for deriving our direct sampling method. Recall that $\boldsymbol H=\frac1{i\omega\mu}\nabla\times\boldsymbol E$, $\boldsymbol H_0=\frac1{i\omega\mu}\nabla\times\boldsymbol  E_0$. 
\begin{theorem}\label{th:integral-Hs}
Define the perturbation field $\boldsymbol H^s:=\boldsymbol H-\boldsymbol H_0$. Then for $\boldsymbol x\in\mathbb R^3\setminus D$, we have
\begin{equation}\label{eq:integral-Hs}
    \boldsymbol H^s(\boldsymbol x)=\int_D\nabla_{\boldsymbol x}G(\boldsymbol x,\boldsymbol y)\times(\nabla_{\boldsymbol y}\times\boldsymbol H^s(\boldsymbol y))d\boldsymbol y,
\end{equation}
where $G(\boldsymbol x,\boldsymbol y)=\frac1{4\pi|\boldsymbol x-\boldsymbol y|}$ is the fundamental solution of the Laplace equation in free space.
\end{theorem}
See \cite{AMMARI2014} for the proof of the above theorem. 
Note that 
\begin{equation*}
    \nabla_{\boldsymbol y}\times\boldsymbol H^s(\boldsymbol y)=\left(\boldsymbol{J}_0(\boldsymbol y)+\sigma\boldsymbol E(\boldsymbol y)\right)-\boldsymbol{J}_0(\boldsymbol y)=\sigma\boldsymbol E(\boldsymbol y),
\end{equation*}
so that \eqref{eq:integral-Hs} can be simplified to
\begin{equation}\label{eq:int}
    \begin{aligned}
    \boldsymbol H^s(\boldsymbol x)=&\int_D\nabla_{\boldsymbol x}G(\boldsymbol x,\boldsymbol y)\times(\sigma\boldsymbol E)(\boldsymbol y)d\boldsymbol y\\
    =&\int_D|(\sigma\boldsymbol E)(\boldsymbol y)|\left(\nabla_{\boldsymbol x}G(\boldsymbol x,\boldsymbol y)\times\widehat{\boldsymbol E}(\boldsymbol y)\right)d\boldsymbol y,
    \end{aligned}
\end{equation}
where $\widehat{\boldsymbol E}$ is the normalization of $\boldsymbol E$, i.e., $\widehat{\boldsymbol E}=\frac{\boldsymbol E}{|\boldsymbol E|}$. 

Suppose we record magnetic field data on a smooth surface $\Gamma$ outside $D$. For simplicity, we shall choose $\Gamma=\partial B_R$, where $B_R$ denotes the open sphere with radius $R$ centered at the origin. We further assume that the radius $R$ is large enough so that the distance between $D$ and $\Gamma$ is greater than a positive constant $d$, i.e., $|\boldsymbol x-\boldsymbol y|>d>0$ for any $\boldsymbol x\in\Gamma$, $\boldsymbol y\in D$. 

The inverse problem considered herein aims to reconstruct the conductive area $D$ from magnetic near-field data $\boldsymbol H|_{\Gamma}$. For this purpose, we propose a direct sampling method for the MIT problem, which computes an index function $I:\Omega\to[0,1]$ for every measurement $\boldsymbol{\mathcal{M}}=\boldsymbol H|_{\Gamma}$. In the following sections, we assume that the sampling domain $\Omega$ is chosen to fully contain the unknown conductive region $D$, i.e., $D\subset\Omega$. 

\section{Direct sampling method for MIT}\label{sect3}
In this section, we begin by introducing the dual product $\langle\cdot,\cdot\rangle_\gamma$ and its crucial decay property, which lays the theoretical foundation for our algorithm. Next, we define a family of point spread functions $K$ and analyze how the parameter $\gamma$ controls their attenuation. Then we discuss the selection of vector $\boldsymbol\beta$, and finally present the complete procedure of our direct sampling algorithm.

\subsection{Point spread functions}
First, we define the following duality product $\langle \cdot,\cdot\rangle_\gamma$ on $\Gamma$, which is mainly motivated by \cite{Chow2014}: 
\begin{equation}\label{eq:product}
    \begin{aligned}
    \langle\boldsymbol a,\boldsymbol b\rangle_\gamma:=&
    \langle\boldsymbol a,(-\Delta_\Gamma)^\gamma \boldsymbol b\rangle_{L^2(\Gamma)}\\
    =&\int_\Gamma\boldsymbol  a\cdot\overline{(-\Delta_\Gamma)^\gamma \boldsymbol b}\,d\Gamma,
    \quad\text{for}\quad 
    \boldsymbol a\in \boldsymbol L^2(\Gamma),\,
    \boldsymbol b\in \boldsymbol H^{2\gamma}(\Gamma)
    \end{aligned}
\end{equation}
with $\gamma\in2\mathbb N$, where $\mathbb{N}$ is the set of non-negative integers. Here $\Delta_\Gamma$ is the Laplace-Beltrami operator and is applied to each component of vector $\boldsymbol b$, i.e., $(-\Delta_\Gamma)^\gamma \boldsymbol b:=((-\Delta_\Gamma )^\gamma b_1, \cdots, (-\Delta_\Gamma)^\gamma b_n)$. Note that in $\boldsymbol H^{2\gamma}(\Gamma)$ this product is self-adjoint and semi-positive definite, thus inducing a semi-norm, denoted by $|\cdot|_\gamma$. 

For the spherical surface $\Gamma=\partial B_R$ in $\mathbb{R}^3$, the Laplace-Beltrami operator $\Delta_\Gamma$ admits a simple expression in spherical coordinates \cite[P28]{SHAPES}, 
\begin{equation*}
    \Delta_\Gamma=\Delta-\frac{\partial^2}{\partial r^2}-\frac2r\frac{\partial}{\partial r},
\end{equation*}
where $\Delta$ is the Laplace operator in $\mathbb{R}^3$. This means we can compute $(-\Delta_\Gamma)^\gamma$ using its analytical form instead of numerical differentiation, which significantly improves computational efficiency and avoids potential errors.

Now we discuss the duality product \eqref{eq:product} of a specific class of functions. Consider the fundamental solution of Laplace equation and its gradient
\begin{equation*}
    G(\boldsymbol x, \boldsymbol y)
    =\frac1{4\pi|\boldsymbol x - \boldsymbol y|},\quad
    \nabla_{\boldsymbol x}G(\boldsymbol x, \boldsymbol y)
    =-\frac{\boldsymbol x - \boldsymbol y}
    {4\pi|\boldsymbol x - \boldsymbol y|^3}.
\end{equation*}
Clearly $\nabla_{\boldsymbol x}G(\cdot,\boldsymbol y)\in \boldsymbol H^{2\gamma}(\Gamma)$ for any given $\boldsymbol y\in B_R$. Let $S=\{\boldsymbol\alpha\in\mathbb C^3||\boldsymbol\alpha|=1\}$ be the set of all complex unit vectors. Define the point spread functions 
\begin{equation}\label{funcK}
    \begin{aligned}
        K_{(\boldsymbol y,\boldsymbol \alpha)}(\boldsymbol z,\boldsymbol\beta)
        :=&\frac{\langle \nabla_{\boldsymbol x}G(\cdot, \boldsymbol y)\times\boldsymbol\alpha, \nabla_{\boldsymbol x}G(\cdot, \boldsymbol z)\times\boldsymbol\beta \rangle_\gamma}
        {| \nabla_{\boldsymbol x}G(\cdot, \boldsymbol z)\times\boldsymbol\beta |_\gamma}\\
        =&\cfrac{\int_\Gamma\left(-\frac{\boldsymbol x - \boldsymbol y}
        {4\pi|\boldsymbol x - \boldsymbol y|^3}\times\boldsymbol{\alpha}\right)\cdot\left((-\Delta_\Gamma)^{\gamma}\left(-\frac{\boldsymbol x - \boldsymbol z}
        {4\pi|\boldsymbol x - \boldsymbol z|^3}\right)\times\bar{\boldsymbol{\beta}}\right)d\boldsymbol{x}}{\left(\int_\Gamma\left|(-\Delta_\Gamma)^{\gamma/2}\left(-\frac{\boldsymbol x - \boldsymbol y}
        {4\pi|\boldsymbol x - \boldsymbol y|^3}\right)\times\boldsymbol{\beta}\right|^2d\boldsymbol{x}\right)^{1/2}}
    \end{aligned}
\end{equation}
for $\boldsymbol y,\boldsymbol{z}\in B_R$ and unit vectors $\boldsymbol\alpha, \boldsymbol\beta\in S$, where all operators $(-\Delta_\Gamma)$ in $\langle\cdot,\cdot\rangle_\gamma$ and $|\cdot|_\gamma$ are applied to the variable $\boldsymbol{x}$. The following proposition ensures that $K$ is well defined. 

\begin{prop}\label{prop:non-zero}
The denominator of point spread functions $K_{(\boldsymbol y,\boldsymbol \alpha)}$ in \eqref{funcK} is non-zero, i.e., $| \nabla_{\boldsymbol x}G(\cdot, \boldsymbol z)\times\boldsymbol\beta |_\gamma\ne 0$ for any $\boldsymbol{z}\in B_R$, $\boldsymbol{\beta}\in S$ and $\gamma\in2\mathbb N$. 
\end{prop}
\begin{proof}
    The case $\gamma=0$ is obvious. For $\gamma\ge 2$, we prove the result by contradiction. Assume
    \begin{equation*}
        | \nabla_{\boldsymbol x}G(\cdot, \boldsymbol z)\times\boldsymbol\beta |_\gamma=\left(\int_\Gamma\left|(-\Delta_\Gamma)^{\gamma/2}\left(\nabla_{\boldsymbol x}G(\boldsymbol{x}, \boldsymbol z)\times\boldsymbol\beta\right)\right|^2d\boldsymbol{x}\right)^{1/2}=0,
    \end{equation*}
    then $(-\Delta_\Gamma)^{\gamma/2}\left(\nabla_{\boldsymbol x}G(\boldsymbol{x}, \boldsymbol z)\times\boldsymbol\beta\right)\overset{a.e.}{=}\boldsymbol{0}$. Since $-\Delta_\Gamma$ is self-adjoint, we have $\text{Null}((-\Delta_\Gamma)^{\gamma/2})=\text{Null}(-\Delta_\Gamma)$, which gives
    \begin{equation*}
        (-\Delta_\Gamma)\left(\nabla_{\boldsymbol x}G(\boldsymbol{x}, \boldsymbol z)\times\boldsymbol\beta\right)\overset{a.e.}{=}\boldsymbol{0}.
    \end{equation*}
    However, according to \cite[P125]{Jost2017}, constant functions are the only functions in $\text{Null}(-\Delta_\Gamma)$ when $\Gamma$ is a compact Riemannian manifold. This contradiction proves the proposition. 
\end{proof}

We then proceed to prove several fundamental properties of $K$, 
which will show that $K$ takes relatively large values near $(\boldsymbol y,\boldsymbol\alpha)$ and decays to zero when $\boldsymbol{z}$ is close to $\Gamma$. 

\begin{prop}
    $K_{(\boldsymbol y,\boldsymbol\alpha)}(\boldsymbol z,\boldsymbol\beta)$ is bounded for all $\boldsymbol{z}\in B_R$. Moreover, it attains its maximum at $(\boldsymbol z,\boldsymbol\beta)=(\boldsymbol y,\boldsymbol\alpha)$.
\end{prop}
\begin{proof}
    By the Cauchy-Schwarz inequality, we have
    \begin{equation*}
        \left|\langle \nabla_{\boldsymbol x}G(\cdot, \boldsymbol y)\times\boldsymbol\alpha, \nabla_{\boldsymbol x}G(\cdot, \boldsymbol z)\times\boldsymbol\beta \rangle_\gamma\right|
        \le|\nabla_{\boldsymbol x}G(\cdot, \boldsymbol y)\times\boldsymbol\alpha|_\gamma|\nabla_{\boldsymbol x}G(\cdot, \boldsymbol z)\times\boldsymbol\beta|_\gamma,
    \end{equation*}
    therefore,
    \begin{equation*}
        \left|K_{(\boldsymbol y,\boldsymbol\alpha)}(\boldsymbol z,\boldsymbol\beta)\right|=\frac{\left|\langle \nabla_{\boldsymbol x}G(\cdot, \boldsymbol y)\times\boldsymbol\alpha, \nabla_{\boldsymbol x}G(\cdot, \boldsymbol z)\times\boldsymbol\beta \rangle_\gamma\right|}{|\nabla_{\boldsymbol x}G(\cdot, \boldsymbol z)\times\boldsymbol\beta|_\gamma}\le|\nabla_{\boldsymbol x}G(\cdot, \boldsymbol y)\times\boldsymbol\alpha|_\gamma.
    \end{equation*}
     The equality holds at $(\boldsymbol z,\boldsymbol\beta)=(\boldsymbol y,\boldsymbol\alpha)$, where $K_{(\boldsymbol y,\boldsymbol\alpha)}(\boldsymbol y,\boldsymbol\alpha)=|\nabla_{\boldsymbol x}G(\cdot, \boldsymbol y)\times\boldsymbol\alpha|_\gamma$ precisely reaches the maximum. 
\end{proof}

We now derive the decaying property of $K$. Without loss of generality, we shall choose the direction of $\boldsymbol{z}$ as the $e_1$ axis of the Cartesian coordinate system so that $\boldsymbol{z}=(z_1,0,0)$. Therefore, we get
\begin{equation*}\label{eq:axis}
    \nabla_{\boldsymbol{x}}G(\boldsymbol{x},\boldsymbol{z})=-\frac{\boldsymbol{x}-(z_1,0,0)}{4\pi|\boldsymbol{x}-(z_1,0,0)|^3}=-\frac{(x_1-z_1,x_2,x_3)}{4\pi\left((x_1-z_1)^2+x_2^2+x_3^2\right)^{3/2}}.
\end{equation*}
For simplicity, we denote the components of $\nabla_{\boldsymbol{x}}G(\boldsymbol{x},\boldsymbol{z})$ by $g_1,g_2,g_3$, that is,
\begin{equation}\label{gi}
    \left\{\begin{aligned}
        g_1=&-\frac{x_1-z_1}{4\pi\left((x_1-z_1)^2+x_2^2+x_3^2\right)^{3/2}},\\
        g_2=&-\frac{x_2}{4\pi\left((x_1-z_1)^2+x_2^2+x_3^2\right)^{3/2}},\\
        g_3=&-\frac{x_3}{4\pi\left((x_1-z_1)^2+x_2^2+x_3^2\right)^{3/2}}.
    \end{aligned}
    \right.
\end{equation}
We will mainly focus on the limit of $K$ as $\boldsymbol{z}$ approaches surface $\Gamma$, which makes the integrals in both the numerator and denominator of $K$ singular. The following lemmas shall provide appropriate estimates for these two singular integrals. 

\begin{lemma}\label{lemma1}
    For $g_i, i=1,2,3$ defined in \eqref{gi}, there exist $\lambda_1>0$ and $C>0$ such that
    \begin{equation}\label{eq:lemma1}
        \int_\Gamma\left((-\Delta_\Gamma)^ng_i\right)^2 d\boldsymbol{x}\ge \frac{C\lambda_1^{2n}}{\left(R^2-|\boldsymbol{z}|^2\right)^2}
    \end{equation}
    for all $n\in\mathbb{N}$. 
\end{lemma}
\begin{proof}
    By \cite[P125]{Jost2017}, the operator $(-\Delta_\Gamma)$ has nonnegative real eigenvalues $\lambda_0<\lambda_1<\lambda_2<\cdots$. Except for the eigenvalue $\lambda_0=0$ realized for a constant as its eigenfunction, all eigenvalues are positive and $\lim_{k\to\infty}\lambda_k=\infty$. For any $f\in L^2(\Gamma)$, we have the orthogonal decomposition
    \begin{equation*}
        f=\sum_{i=0}^\infty(f,v_i)v_i,
    \end{equation*}
    where $v_i$ is the eigenfunction of $\lambda_i$ and $(\cdot,\cdot)$ is the inner product of $L^2(\Gamma)$. Hence, for $f\in H^{2n}(\Gamma)$ and $n\ge 1$, we have
    \begin{equation*}
        \begin{aligned}
            \int_\Gamma\left((-\Delta_\Gamma)^nf\right)^2 d\boldsymbol{x}=\int_\Gamma\left((-\Delta_\Gamma)^n(f-(f,1)1)\right)^2 d\boldsymbol{x}\ge\lambda_1^{2n}\int_\Gamma \left(f-(f,1)1\right)^2d\boldsymbol{x}.
        \end{aligned}
    \end{equation*}
    It is easy to verify by calculation that $(g_i,1)=\int_\Gamma g_id\boldsymbol{x}=0$ holds for $i=1,2,3$, so that
    \begin{equation*}
        \begin{aligned}
            \int_\Gamma\left((-\Delta_\Gamma)^ng_1\right)^2 d\boldsymbol{x}\ge&\lambda_1^{2n}\int_\Gamma g_1^2d\boldsymbol{x}=\lambda_1^{2n}\cdot\tfrac{R}{32\pi|\boldsymbol{z}|^3}\begin{pmatrix}\ln\left(\frac{R+|\boldsymbol{z}|}{ R-|\boldsymbol{z}|}\right)-\frac{4R|\boldsymbol{z}|}{R^2-|\boldsymbol{z}|^2}+\frac{2R|\boldsymbol{z}|\left(R^2+|\boldsymbol{z}|^2\right)}{\left(R^2-|\boldsymbol{z}|^2\right)^2}\end{pmatrix},\\
            \int_\Gamma\left((-\Delta_\Gamma)^ng_2\right)^2 d\boldsymbol{x}\ge&\lambda_1^{2n}\int_\Gamma g_2^2d\boldsymbol{x}=\lambda_1^{2n}\cdot\tfrac{R}{32\pi|\boldsymbol{z}|^3}\begin{pmatrix}\frac12\ln\left(\frac{R-|\boldsymbol{z}|}{ R+|\boldsymbol{z}|}\right)+\frac{R|\boldsymbol{z}|\left(R^2+|\boldsymbol{z}|^2\right)}{\left(R^2-|\boldsymbol{z}|^2\right)^2}\end{pmatrix},\\
            \int_\Gamma\left((-\Delta_\Gamma)^ng_3\right)^2 d\boldsymbol{x}\ge&\lambda_1^{2n}\int_\Gamma g_3^2d\boldsymbol{x}=\lambda_1^{2n}\cdot\tfrac{R}{32\pi|\boldsymbol{z}|^3}\begin{pmatrix}\frac12\ln\left(\frac{R-|\boldsymbol{z}|}{ R+|\boldsymbol{z}|}\right)+\frac{R|\boldsymbol{z}|\left(R^2+|\boldsymbol{z}|^2\right)}{\left(R^2-|\boldsymbol{z}|^2\right)^2}\end{pmatrix}.
        \end{aligned}
    \end{equation*}
    All the above expressions have the domination term $1/\left(R^2-|\boldsymbol{z}|^2\right)^2$ as $|\boldsymbol{z}|\to R$. Therefore, there exists $C>0$ such that
    \begin{equation*}
        \int_\Gamma\left((-\Delta_\Gamma)^ng_i\right)^2 d\boldsymbol{x}\ge \frac{C\lambda_1^{2n}}{\left(R^2-|\boldsymbol{z}|^2\right)^2}
    \end{equation*}
    for all $n\in\mathbb{N}$ and $i=1,2,3$. 
\end{proof}

\begin{lemma}\label{lemma2}
    For any $\gamma\in 2\mathbb{N}$,
    $$
        |\nabla_{\boldsymbol{x}}G(\cdot,\boldsymbol{z})\times\boldsymbol{\beta}|_\gamma^{-1}=O\left(R^2-|\boldsymbol{z}|^2\right),\quad |\boldsymbol{z}|\to R.
    $$
\end{lemma}
\begin{proof}
    By definition of $ K_{(\boldsymbol y,\boldsymbol \alpha)}(\boldsymbol z,\boldsymbol\beta)$ in \eqref{funcK}, we have
    \begin{equation}\label{eq:lemma2.1}
        \begin{aligned}
            &|\nabla_{\boldsymbol{x}}G(\cdot,\boldsymbol{z})\times\boldsymbol{\beta}|_\gamma^2
            =\int_\Gamma\left|\left(-\Delta_\Gamma\right)^{\gamma/2}\left(g_1,g_2,g_3\right)\right|^2d\boldsymbol{x}\\
            &\qquad-\boldsymbol{\beta}^T\left(\int_\Gamma\left(\left(-\Delta_\Gamma\right)^{\gamma/2}\left(g_1,g_2,g_3\right)\right)\left(\left(-\Delta_\Gamma\right)^{\gamma/2}\left(g_1,g_2,g_3\right)\right)^Td\boldsymbol{x}\right)\bar{\boldsymbol{\beta}}.
        \end{aligned}
    \end{equation}
    The operator $(-\Delta_\Gamma)$ is rotation invariant, so that $(-\Delta_\Gamma)^{\gamma/2}\left(\nabla_{\boldsymbol{x}}G(\boldsymbol{x},\boldsymbol{z})\right)$ preserves the symmetry property of $\nabla_{\boldsymbol{x}}G(\boldsymbol{x},\boldsymbol{z})$. Consequently, $(-\Delta_\Gamma)^{\gamma/2}g_2$ and $(-\Delta_\Gamma)^{\gamma/2}g_3$ are odd functions of $x_2$ and $x_3$, respectively. This property ensures that
    \begin{equation*}
        \int_\Gamma\left(\left(-\Delta_\Gamma\right)^{\gamma/2}g_i\right)\left(\left(-\Delta_\Gamma\right)^{\gamma/2}g_j\right)d\boldsymbol{x}=0
    \end{equation*}
    when $i\ne j$. Thus \eqref{eq:lemma2.1} can be simplified to
    \begin{equation*}
        \begin{aligned}
            |\nabla_{\boldsymbol{x}}G(\cdot,\boldsymbol{z})\times\boldsymbol{\beta}|_\gamma^2
            =&\sum_{i=1}^3\left(1-\beta_i\bar\beta_i\right)\left(\int_\Gamma\left(\left(-\Delta_\Gamma\right)^{\gamma/2}g_i\right)^2d\boldsymbol{x}\right)\\
            \ge&\left(\sum_{i=1}^3\left(1-\beta_i\bar\beta_i\right)\right)\left(\min_{1\le i\le 3}\int_\Gamma\left(\left(-\Delta_\Gamma\right)^{\gamma/2}g_i\right)^2d\boldsymbol{x}\right)\\
            =&2\min_{1\le i\le 3}\int_\Gamma\left(\left(-\Delta_\Gamma\right)^{\gamma/2}g_i\right)^2d\boldsymbol{x}.
        \end{aligned}
    \end{equation*}
    Using the estimation \eqref{eq:lemma1} in Lemma \ref{lemma1}, we know that
    \begin{equation*}
        \begin{aligned}
            |\nabla_{\boldsymbol{x}}G(\cdot,\boldsymbol{z})\times\boldsymbol{\beta}|_\gamma^2
            \ge\frac{2C\lambda_1^\gamma}{\left(R^2-|\boldsymbol{z}|^2\right)^2},
        \end{aligned}
    \end{equation*}
    which proves the lemma.
\end{proof}

With the previous two lemmas, we can derive the main estimate for the point spread function $K$.
\begin{theorem}\label{thm:decay}
    For any fixed $\gamma\in2\mathbb{N}$, $\boldsymbol{y}\in B_R$ and $\boldsymbol{\alpha}\in S$, we have
    \begin{equation}\label{eq:decay}
        \lim_{|\boldsymbol{z}|\to R}K_{(\boldsymbol y,\boldsymbol\alpha)}(\boldsymbol z,\boldsymbol\beta)=0.
    \end{equation}
    This limit is uniform for all $\boldsymbol{\beta}\in S$.
\end{theorem}
\begin{proof}
    Since the operator $(-\Delta_\Gamma)$ is self-adjoint, we have the following estimation for the numerator of $K$, 
    \begin{equation*}
        \begin{aligned}
            &\left| \langle \nabla_{\boldsymbol x}G(\cdot, \boldsymbol y)\times\boldsymbol\alpha, \nabla_{\boldsymbol x}G(\cdot, \boldsymbol z)\times\boldsymbol\beta \rangle_{\gamma} \right|\\
            =&\left|\int_\Gamma\left((-\Delta_\Gamma)^{\gamma}\left(-\frac{\boldsymbol{x}-\boldsymbol y}{4\pi|\boldsymbol{x}-\boldsymbol y|^3}\right)\times\boldsymbol{\alpha}\right)\cdot\left(-\frac{\boldsymbol{x}-\boldsymbol{z}}{4\pi|\boldsymbol{x}-\boldsymbol{z}|^3}\times\bar{\boldsymbol{\beta}}\right)d\boldsymbol{x}\right|\\
            \le&\,C_1\int_\Gamma\left|-\frac{\boldsymbol{x}-\boldsymbol{z}}{4\pi|\boldsymbol{x}-\boldsymbol{z}|^3}\times\bar{\boldsymbol{\beta}}\right|d\boldsymbol{x}
            \le C_1\int_\Gamma\frac1{|\boldsymbol{x}-\boldsymbol{z}|^2}d\boldsymbol{x}\\
            =&\frac{C_1}{R|\boldsymbol{z}|}\ln\left(\frac{R+|\boldsymbol{z}|}{R-|\boldsymbol{z}|}\right).
        \end{aligned}
    \end{equation*}
    Combined with Lemma \ref{lemma2} we get 
    \begin{equation*}
        \left|K_{(\boldsymbol{y},\boldsymbol{\alpha})}(\boldsymbol{z},\boldsymbol{\beta})\right|=O\left(\left(R^2-|\boldsymbol{z}|^2\right)\ln\left(\frac{1}{R-|\boldsymbol{z}|}\right)\right)=o(1),\quad |\boldsymbol{z}|\to R,
    \end{equation*}
    thus completing the proof.
\end{proof}


To illustrate the aforementioned property in Theorem \ref{thm:decay} much more clearly, we now consider a specific case and derive an explicit estimate for the function $K$. 

\begin{prop}\label{prop:decay}
Let $\gamma = 0$ and $\boldsymbol{y} = 0$. There exist constants $c, C>0$ such that
\begin{equation*}
    c\frac{|\boldsymbol{\alpha}\cdot\bar{\boldsymbol{\beta}}|}{\sqrt{D_0}}\le \left|K_{(\boldsymbol 0, \boldsymbol \alpha)}(\boldsymbol z, \boldsymbol\beta)\right|\le C\frac{|\boldsymbol{\alpha}\cdot\bar{\boldsymbol{\beta}}|}{\sqrt{D_0}}
\end{equation*}
with $D_0=\frac{R^2}{4\pi\left(R^2-|\boldsymbol{z}|^2\right)^2}$. 
\end{prop}
\begin{proof}
Consider the specific case where $\gamma=0$ and $\boldsymbol{y}=0$. We compute $K_{(\boldsymbol 0, \boldsymbol \alpha)}(\boldsymbol z, \boldsymbol\beta)$ to demonstrate its decay property. The denominator is given by 
\begin{equation}\label{ex1:deno}
\begin{aligned}
    |\nabla_{\boldsymbol{x}} G(\cdot,\boldsymbol{z})\times\boldsymbol{\beta}|_0^2=&\int_\Gamma\left|-\frac{\boldsymbol{x}-\boldsymbol{z}}{4\pi|\boldsymbol{x}-\boldsymbol{z}|^3}\times\boldsymbol{\beta}\right|^2d\boldsymbol{x}\\
    =&\left(\int_\Gamma\frac1{16\pi^2}\frac1{|\boldsymbol{x}-\boldsymbol{z}|^4}d\boldsymbol{x}\right)(\boldsymbol{\beta}\cdot\bar{\boldsymbol{\beta}})-\boldsymbol{\beta}^T\left(\int_\Gamma\frac{(\boldsymbol{x}-\boldsymbol{z})(\boldsymbol{x}-\boldsymbol{z})^T}{16\pi^2|\boldsymbol{x}-\boldsymbol{z}|^4}d\boldsymbol{x}\right)\bar{\boldsymbol{\beta}}\\
    =&\frac{R^2}{4\pi \left(R^2-|\boldsymbol{z}|^2\right)^2}-\boldsymbol{\beta}^T\left(\int_\Gamma\frac{(\boldsymbol{x}-\boldsymbol{z})(\boldsymbol{x}-\boldsymbol{z})^T}{16\pi^2|\boldsymbol{x}-\boldsymbol{z}|^4}d\boldsymbol{x}\right)\bar{\boldsymbol{\beta}}.
\end{aligned}
\end{equation}
To simplify the integral, we align $\boldsymbol{z}$ with the $\boldsymbol{e_1}$ axis of the Cartesian coordinate system, such that the $\boldsymbol{e_2}O\boldsymbol{e_3}$ plane is perpendicular to $\boldsymbol{z}$. In this coordinate system, the integral above can be expressed in matrix form as
\begin{equation*}
\begin{aligned}
    &|\nabla_{\boldsymbol{x}} G(\cdot,\boldsymbol{z})\times\boldsymbol{\beta}|_0^2=\frac{R^2}{4\pi \left(R^2-|\boldsymbol{z}|^2\right)^2}-\frac{R}{32\pi|\boldsymbol{z}|^3}\boldsymbol{\beta}^T
    \begin{bmatrix}
    d_1&0&0\\0&d_2&0\\0&0&d_3
    \end{bmatrix}
    \bar{\boldsymbol{\beta}},
\end{aligned}
\end{equation*}
where
\begin{equation*}
    \begin{aligned}
    d_1=&\,\ln\left(\tfrac{R+|\boldsymbol{z}|}{ R-|\boldsymbol{z}|}\right)-\tfrac{4R|\boldsymbol{z}|}{R^2-|\boldsymbol{z}|^2}+\tfrac{2R|\boldsymbol{z}|\left(R^2+|\boldsymbol{z}|^2\right)}{\left(R^2-|\boldsymbol{z}|^2\right)^2},\\
    d_2=&\,d_3=\tfrac12\ln\left(\tfrac{R-|\boldsymbol{z}|}{ R+|\boldsymbol{z}|}\right)+\tfrac{R|\boldsymbol{z}|\left(R^2+|\boldsymbol{z}|^2\right)}{\left(R^2-|\boldsymbol{z}|^2\right)^2}.
    \end{aligned}
\end{equation*}
Note that the entries in the diagonal matrix are all bounded by the domination term $D_0=\frac{R^2}{4\pi\left(R^2-|\boldsymbol{z}|^2\right)^2}$, i.e., 
\begin{equation*}
    \begin{aligned}
        \frac13 D_0
        \le\frac{R}{32\pi|\boldsymbol{z}|^3}d_1
        \le\frac12 D_0,\quad
        \frac14 D_0
        \le\frac{R}{32\pi|\boldsymbol{z}|^3}d_{2,3}
        \le\frac13 D_0.
    \end{aligned}
\end{equation*}
Consequently, for all $\boldsymbol{\beta}\in S$, the denominator \eqref{ex1:deno} of $K$ is uniformly bounded by
\begin{equation}\label{ex1:deno_bound}
    \frac12\frac{R^2}{4\pi\left(R^2-|\boldsymbol{z}|^2\right)^2}
    \le|\nabla_{\boldsymbol{x}} G(\cdot,\boldsymbol{z})\times\boldsymbol{\beta}|_0^2
    \le\frac23\frac{R^2}{4\pi\left(R^2-|\boldsymbol{z}|^2\right)^2}.
\end{equation}
Meanwhile, the numerator is computed as
\begin{equation}\label{ex1:nume}
    \begin{aligned}
        \langle \nabla_{\boldsymbol x}G&(\cdot, \boldsymbol 0)\times\boldsymbol\alpha, \nabla_{\boldsymbol x}G(\cdot, \boldsymbol z)\times\boldsymbol\beta \rangle_0\\
        =&\int_\Gamma\left(-\frac{\boldsymbol{x}}{4\pi|\boldsymbol{x}|^3}\times\boldsymbol{\alpha}\right)\cdot\left(-\frac{\boldsymbol{x}-\boldsymbol{z}}{4\pi|\boldsymbol{x}-\boldsymbol{z}|^3}\times\bar{\boldsymbol{\beta}}\right)\\
        =&\left(\frac1{16\pi^2}\int_\Gamma\frac{\boldsymbol{x}\cdot(\boldsymbol{x}-\boldsymbol{z})}{|\boldsymbol{x}|^3|\boldsymbol{x-z}|^3}d\boldsymbol{x}\right)(\boldsymbol{\alpha}\cdot\bar{\boldsymbol{\beta}})-\boldsymbol{\alpha}^T\left(\int_\Gamma\frac{\boldsymbol{x}(\boldsymbol{x}-\boldsymbol{z})^T}{16\pi^2|\boldsymbol{x}|^3|\boldsymbol{x-z}|^3}d\boldsymbol{x}\right)\bar{\boldsymbol{\beta}}\\
        =&\frac{1}{4\pi R^2}(\boldsymbol{\alpha}\cdot\bar{\boldsymbol{\beta}})-\boldsymbol{\alpha}^T\left(\frac{1}{12\pi R^2}I\right)\bar{\boldsymbol{\beta}}\\
        =&\frac{1}{6\pi R^2}(\boldsymbol{\alpha}\cdot\bar{\boldsymbol{\beta}}).
    \end{aligned}
\end{equation}
Therefore, by \eqref{ex1:deno_bound} and \eqref{ex1:nume} we obtain the following bounds for $K_{(\boldsymbol 0, \boldsymbol \alpha)}(\boldsymbol z, \boldsymbol\beta)$: 
\begin{equation*}\label{ex1:K}
    \sqrt{\frac32}\frac{|\boldsymbol{\alpha}\cdot\bar{\boldsymbol{\beta}}|}{6\pi R^2\sqrt{D_0}}\le\left|K_{(\boldsymbol 0, \boldsymbol \alpha)}(\boldsymbol z, \boldsymbol\beta)\right|\le\sqrt{2}\frac{|\boldsymbol{\alpha}\cdot\bar{\boldsymbol{\beta}}|}{6\pi R^2\sqrt{D_0}}.
\end{equation*}
This shows that $K_{(\boldsymbol 0, \boldsymbol \alpha)}(\boldsymbol z, \boldsymbol\beta)$ scales as $D_0^{-1/2}$ and vanishes for all $\boldsymbol{\beta}\in S$ as $|\boldsymbol{z}|\to R$.
\end{proof}

\subsection{Influence of parameter $\gamma$}

We observe that a larger $\gamma$ will give a sharper peak of $K_{(\boldsymbol y,\boldsymbol\alpha)}$ at $(\boldsymbol z,\boldsymbol\beta)=(\boldsymbol y,\boldsymbol \alpha)$ (see Figure \ref{fig}), where the main contribution of the function $K_{(\boldsymbol y,\boldsymbol\alpha)}$ concentrates in the neighborhood of $\boldsymbol y$ and only takes negligibly small values in more distant regions.

\begin{figure}[htbp]
\centering
    \begin{subfigure}{0.42\linewidth}
        \centering
        \includegraphics[width=0.92\linewidth]{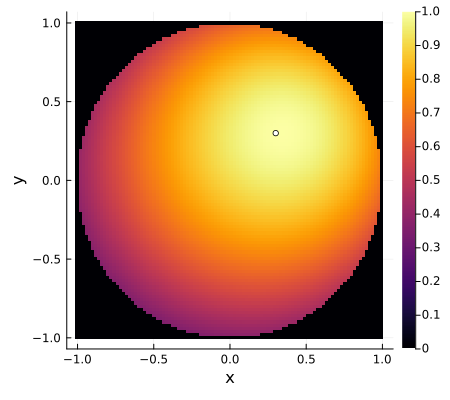}
    \end{subfigure}
    \centering
    \begin{subfigure}{0.42\linewidth}
        \centering
        \includegraphics[width=0.92\linewidth]{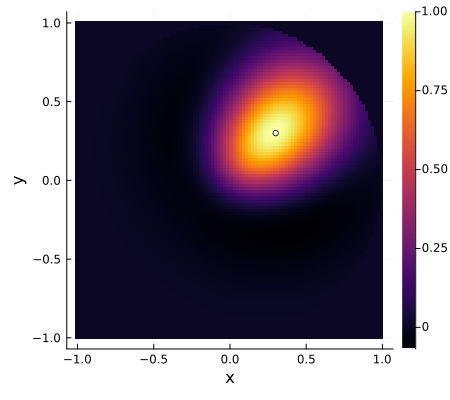}
    \end{subfigure}
\caption{Values of $K_{(\boldsymbol y, \boldsymbol \alpha)}(\cdot, \boldsymbol\beta)$ on cross-section $\{x^2+y^2\le1,\,z=0\}$ with different $\gamma$s. Left: $\gamma=0$. Right: $\gamma=4$. In both cases $\boldsymbol y=(0.3, 0.3, 0.0)$ and $\boldsymbol\alpha=\boldsymbol\beta=(1,0,0)$.}
\label{fig}
\end{figure}

Next, we present a specific case to illustrate the influence of the parameter $\gamma$ on the function $K$. The following result should be compared to Proposition  \ref{prop:decay}.

\begin{prop}\label{prop:decay2}
Let $\gamma = 2$ and $\boldsymbol{y} = 0$. There exist constants $c, C>0$ such that
\begin{equation*}
    c\frac{|\boldsymbol{\alpha}\cdot\bar{\boldsymbol{\beta}}|}{\sqrt{D_2}}\le \left|K_{(\boldsymbol 0, \boldsymbol \alpha)}(\boldsymbol z, \boldsymbol\beta)\right|\le C\frac{|\boldsymbol{\alpha}\cdot\bar{\boldsymbol{\beta}}|}{\sqrt{D_2}}
\end{equation*}
with $D_2=\frac{R^6+12R^4|\boldsymbol{z}|^2+15R^2|\boldsymbol{z}|^4+2|\boldsymbol{z}|^6}{\pi \left(R^2-|\boldsymbol{z}|^2\right)^6}$. 
\end{prop}
\begin{proof}
Consider the specific case where $\gamma=2$ and $\boldsymbol{y}=0$. Similar to the proof of Proposition \ref{prop:decay}, we choose the direction of $\boldsymbol{z}$ as the $e_1$ axis and use the spherical coordinates $\boldsymbol{x}=(r\cos\theta, r\sin\theta\cos\phi, r\sin\theta\sin\phi)$. In this case, we have
\begin{eqnarray*}
        &&(-\Delta_\Gamma)\frac{\boldsymbol{x}}{|\boldsymbol{x}|^3}=\frac{\boldsymbol{x}}{|\boldsymbol{x}|^5},\\
        &&(-\Delta_\Gamma)\frac{x_1-z_1}{|\boldsymbol{x}-\boldsymbol{z}|^3}=\left(\left(9|\boldsymbol{x}||\boldsymbol{z}|^3-6|\boldsymbol{x}|^3|\boldsymbol{z}|\right)+\left(2|\boldsymbol{x}|^4-5|\boldsymbol{x}|^2|\boldsymbol{z}|^2-4|\boldsymbol{z}|^4\right)\cos\theta\frac{}{}\right.\\
        &&\quad\left.\,\frac{}{}+\left(|\boldsymbol{x}||\boldsymbol{z}|^3+4|\boldsymbol{x}|^3|\boldsymbol{z}|\right)\cos^2\theta-|\boldsymbol{x}|^2|\boldsymbol{z}|^2\cos^3\theta\right)/\left({|\boldsymbol{x}|\left(|\boldsymbol{x}|^2+|\boldsymbol{z}|^2-2|\boldsymbol{x}||\boldsymbol{z}|\cos\theta\right)^{7/2}}\right),\\
        &&(-\Delta_\Gamma)\frac{x_2}{|\boldsymbol{x}-\boldsymbol{z}|^3}=\left(\left(2|\boldsymbol{x}|^4-11|\boldsymbol{x}|^2|\boldsymbol{z}|^2+2|\boldsymbol{z}|^4\right)+4|\boldsymbol{x}||\boldsymbol{z}|\left(|\boldsymbol{x}|^2+|\boldsymbol{z}|^2\right)\cos\theta\frac{}{}\right.\\
        &&\quad\left.\,\frac{}{}-|\boldsymbol{x}|^2|\boldsymbol{z}|^2\cos^2\theta\right)\sin\theta\cos\phi\,/\left({|\boldsymbol{x}|\left(|\boldsymbol{x}|^2+|\boldsymbol{z}|^2-2|\boldsymbol{x}||\boldsymbol{z}|\cos\theta\right)^{7/2}}\right),\\
        &&(-\Delta_\Gamma)\frac{x_3}{|\boldsymbol{x}-\boldsymbol{z}|^3}=\left(\left(2|\boldsymbol{x}|^4-11|\boldsymbol{x}|^2|\boldsymbol{z}|^2+2|\boldsymbol{z}|^4\right)+4|\boldsymbol{x}||\boldsymbol{z}|\left(|\boldsymbol{x}|^2+|\boldsymbol{z}|^2\right)\cos\theta\frac{}{}\right.\\
        &&\quad\left.\,\frac{}{}-|\boldsymbol{x}|^2|\boldsymbol{z}|^2\cos^2\theta\right)\sin\theta\sin\phi\,/\left({|\boldsymbol{x}|\left(|\boldsymbol{x}|^2+|\boldsymbol{z}|^2-2|\boldsymbol{x}||\boldsymbol{z}|\cos\theta\right)^{7/2}}\right).
\end{eqnarray*}
The denominator of $K_{(\boldsymbol 0, \boldsymbol \alpha)}(\boldsymbol z, \boldsymbol\beta)$ is given by
\begin{equation}\label{ex2:deno}
\begin{aligned}
    &|\nabla_{\boldsymbol{x}} G(\cdot,\boldsymbol{z})\times\boldsymbol{\beta}|_2^2\\
    =&\int_\Gamma\left|\left(-\Delta_\Gamma\right)\left(-\frac{\boldsymbol{x}-\boldsymbol{z}}{4\pi|\boldsymbol{x}-\boldsymbol{z}|^3}\right)\times\boldsymbol{\beta}\right|^2d\boldsymbol{x}\\
    =&\frac{R^6+12R^4|\boldsymbol{z}|^2+15R^2|\boldsymbol{z}|^4+2|\boldsymbol{z}|^6}{\pi \left(R^2-|\boldsymbol{z}|^2\right)^6}-\frac{1}{512\pi|\boldsymbol{z}|^3R^3}\boldsymbol{\beta}^T
    \begin{bmatrix}
    d_1&0&0\\0&d_2&0\\0&0&d_3
    \end{bmatrix}
    \bar{\boldsymbol{\beta}},
\end{aligned}
\end{equation}
where
\begin{equation*}
    \begin{aligned}
    d_1=&\,\ln\left(\tfrac{R+|\boldsymbol{z}|}{ R-|\boldsymbol{z}|}\right)-\tfrac{2R|\boldsymbol{z}|\left(R^{10}-91R^8|\boldsymbol{z}|^2-1318R^6|\boldsymbol{z}|^4-2086R^4|\boldsymbol{z}|^6-347R^2|\boldsymbol{z}|^8+|\boldsymbol{z}|^{10}\right)}{\left(R^2-|\boldsymbol{z}|^2\right)^6},\\
    d_2=&\,d_3=\tfrac12\ln\left(\tfrac{R-|\boldsymbol{z}|}{ R+|\boldsymbol{z}|}\right)+\tfrac{R|\boldsymbol{z}|\left(R^2+|\boldsymbol{z}|^2\right)\left(R^8+164R^6|\boldsymbol{z}|^2+1590R^4|\boldsymbol{z}|^4+164R^2|\boldsymbol{z}|^6+|\boldsymbol{z}|^{8}\right)}{\left(R^2-|\boldsymbol{z}|^2\right)^6}.
\end{aligned}
\end{equation*}
Entries in the matrix are bounded by the domination term $D_2=\frac{R^6+12R^4|\boldsymbol{z}|^2+15R^2|\boldsymbol{z}|^4+2|\boldsymbol{z}|^6}{\pi \left(R^2-|\boldsymbol{z}|^2\right)^6}$, i.e., 
\begin{equation*}
    \begin{aligned}
        \frac13 D_2
        \le \frac{1}{512\pi|\boldsymbol{z}|^3R^3}d_1
        \le\frac12D_2,\quad
        \frac14D_2
        \le\frac{1}{512\pi|\boldsymbol{z}|^3R^3}d_{2,3}
        \le\frac13D_2.
    \end{aligned}
\end{equation*}
Consequently, for all $\boldsymbol{\beta}\in S$, the denominator \eqref{ex2:deno} of $K$ is uniformly bounded by
\begin{equation}\label{ex2:deno_bound}
    \frac12D_2
    \le|\nabla_{\boldsymbol{x}} G(\cdot,\boldsymbol{z})\times\boldsymbol{\beta}|_2^2
    \le\frac23D_2.
\end{equation}
Meanwhile, the numerator is computed as
\begin{equation}\label{ex2:nume}
    \begin{aligned}
        \langle \nabla_{\boldsymbol x}G&(\cdot, \boldsymbol 0)\times\boldsymbol\alpha, \nabla_{\boldsymbol x}G(\cdot, \boldsymbol z)\times\boldsymbol\beta \rangle_2\\
        =&\int_\Gamma\left(\left(-\Delta_\Gamma\right)\left(-\frac{\boldsymbol{x}}{4\pi|\boldsymbol{x}|^3}\right)\times\boldsymbol{\alpha}\right)\cdot\left(\left(-\Delta_\Gamma\right)\left(-\frac{\boldsymbol{x}-\boldsymbol{z}}{4\pi|\boldsymbol{x}-\boldsymbol{z}|^3}\right)\times\bar{\boldsymbol{\beta}}\right)\\
        =&\frac{1}{\pi R^6}(\boldsymbol{\alpha}\cdot\bar{\boldsymbol{\beta}})-\boldsymbol{\alpha}^T\left(\frac{1}{3\pi R^6}I\right)\bar{\boldsymbol{\beta}}\\
        =&\frac{2}{3\pi R^6}(\boldsymbol{\alpha}\cdot\bar{\boldsymbol{\beta}}).
    \end{aligned}
\end{equation}
Therefore, by \eqref{ex2:deno_bound} and \eqref{ex2:nume} we obtain the following bounds for $K_{(\boldsymbol 0, \boldsymbol \alpha)}(\boldsymbol z, \boldsymbol\beta)$: 
\begin{equation*}\label{ex2:K}
    \sqrt{\frac32}\frac{2|\boldsymbol{\alpha}\cdot\bar{\boldsymbol{\beta}}|}{3\pi R^2\sqrt{D_2}}\le\left|K_{(\boldsymbol 0, \boldsymbol \alpha)}(\boldsymbol z, \boldsymbol\beta)\right|\le\sqrt{2}\frac{2|\boldsymbol{\alpha}\cdot\bar{\boldsymbol{\beta}}|}{3\pi R^2\sqrt{D_2}}.
\end{equation*}
This shows that $K_{(\boldsymbol 0, \boldsymbol \alpha)}(\boldsymbol z, \boldsymbol\beta)$ is of the same growth rate as $D_2^{-1/2}$, and vanishes for all $\boldsymbol{\beta}\in S$ as $|\boldsymbol{z}|\to R$.
\end{proof}
Figure \ref{fig:gammas} shows the graphs of normalized $D_0^{-1/2}$ and $D_2^{-1/2}$, which demonstrate the growth rate of $K_{(\boldsymbol{0},\boldsymbol{\alpha})}(\boldsymbol{z},\boldsymbol{\beta})$ with different $\gamma$s in Proposition \ref{prop:decay} and \ref{prop:decay2}, i.e., $\gamma=0$ and $\gamma=2$, respectively. We can see that with a larger $\gamma$ the maximum value of $K$ will become more pronounced. This enables it to better locate the conductor area during the subsequent imaging process.

\begin{figure}[htbp]
\centering
\includegraphics[width=0.6\textwidth]{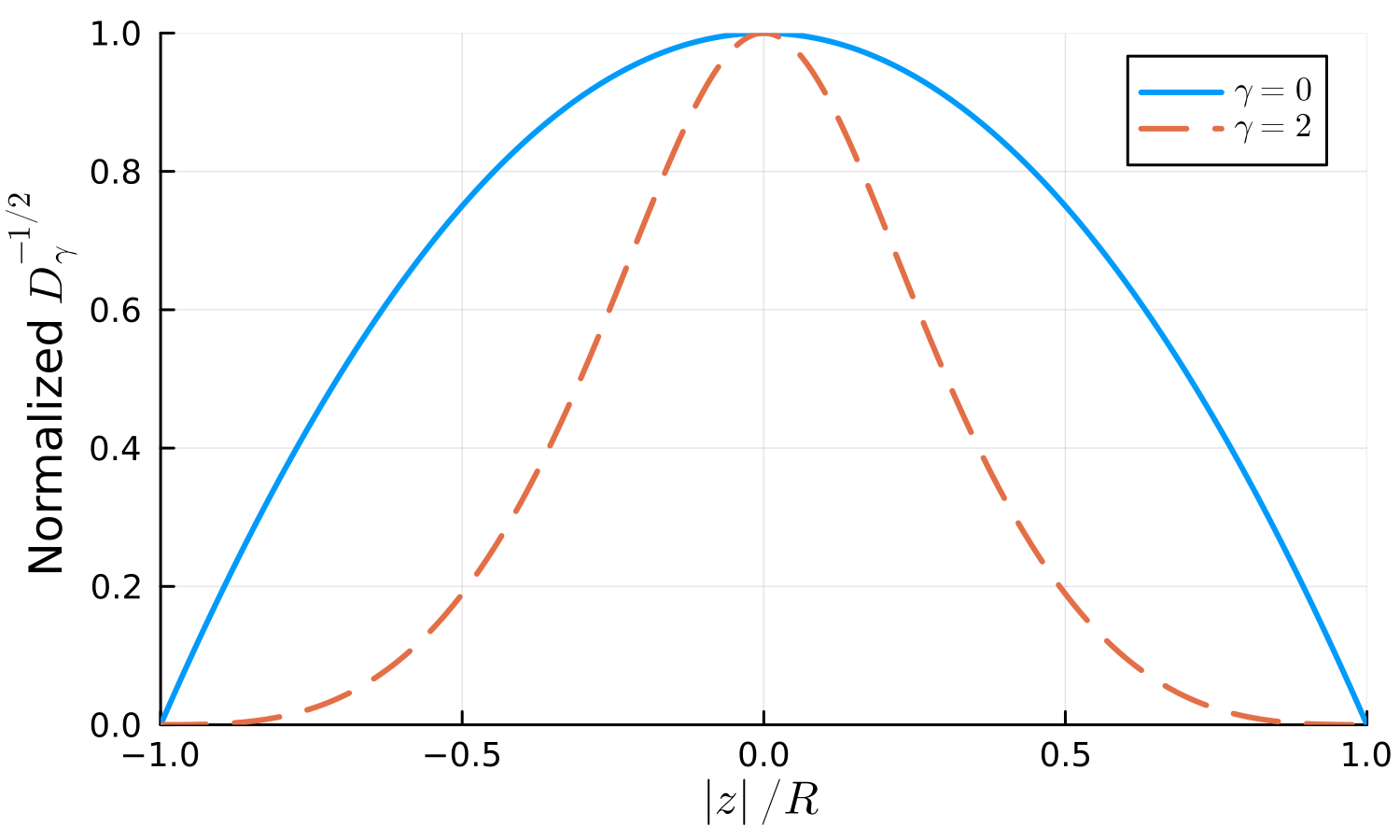}
\caption{Graphs of normalized $D_0^{-1/2}$ and $D_2^{-1/2}$. They show the growth rate of $K_{(\boldsymbol{0},\boldsymbol{\alpha})}(\boldsymbol{z},\boldsymbol{\beta})$ with different $\gamma$s, $\gamma=0$ and $\gamma=2$, respectively.}
\label{fig:gammas}
\end{figure}

\subsection{Main algorithm}

Recall that $\boldsymbol{H}^s(\boldsymbol{x})$ denotes the magnetic perturbation field in Theorem \ref{th:integral-Hs}. For all $\boldsymbol z\in \Omega$ and $\boldsymbol\beta\in S$, we define
\begin{equation}\label{eq:J}
    \begin{aligned}
        J(\boldsymbol z,\boldsymbol\beta):=&\frac{\left|\left\langle\boldsymbol H^s(\cdot), \nabla_{\boldsymbol x}G(\cdot,\boldsymbol z)\times\boldsymbol\beta \right\rangle_\gamma\right|}{|\nabla_{\boldsymbol x}G(\cdot,\boldsymbol z)\times\boldsymbol\beta |_\gamma}\\
        =&\cfrac{\left|\int_\Gamma \boldsymbol{H}^s(\boldsymbol{x})\cdot\left((-\Delta_\Gamma)^{\gamma}\left(-\frac{\boldsymbol x - \boldsymbol z}
        {4\pi|\boldsymbol x - \boldsymbol z|^3}\right)\times\bar{\boldsymbol{\beta}}\right)d\boldsymbol{x}\right|}{\left(\int_\Gamma\left|(-\Delta_\Gamma)^{\gamma/2}\left(-\frac{\boldsymbol x - \boldsymbol y}
        {4\pi|\boldsymbol x - \boldsymbol y|^3}\right)\times\boldsymbol{\beta}\right|^2d\boldsymbol{x}\right)^{1/2}}.
    \end{aligned}
\end{equation}
Using \eqref{eq:int} we get
\begin{equation}\label{eq:J2}
    \begin{aligned}
    J(\boldsymbol z,\boldsymbol\beta)=&\frac{\left|\left\langle\boldsymbol H^s(\cdot), \nabla_{\boldsymbol x}G(\cdot,\boldsymbol z)\times\boldsymbol\beta \right\rangle_\gamma\right|}{\left|\nabla_{\boldsymbol x}G(\cdot,\boldsymbol z)\times\boldsymbol \beta \right|_\gamma}\\
     =&\left|\int_{\boldsymbol y\in D}|(\sigma\boldsymbol E)(\boldsymbol y)|\,\frac{\left\langle\nabla_{\boldsymbol x}G(\cdot,\boldsymbol y)\times\widehat{\boldsymbol E}(\boldsymbol y), \nabla_{\boldsymbol x}G(\cdot,\boldsymbol z)\times\boldsymbol\beta \right\rangle_\gamma}{\left|\nabla_{\boldsymbol x}G(\cdot,\boldsymbol z)\times\boldsymbol\beta \right|_\gamma}\,d\boldsymbol y\right|\\
    \le&\int_{\boldsymbol y\in D}|(\sigma\boldsymbol E)(\boldsymbol y)|\,\frac{\left|\left\langle\nabla_{\boldsymbol x}G(\cdot,\boldsymbol y)\times\widehat{\boldsymbol E}(\boldsymbol y), \nabla_{\boldsymbol x}G(\cdot,\boldsymbol z)\times\boldsymbol\beta \right\rangle_\gamma\right|}{\left|\nabla_{\boldsymbol x}G(\cdot,\boldsymbol z)\times\boldsymbol\beta \right|_\gamma}\,d\boldsymbol y\\
    =&\int_{\boldsymbol y\in D}|(\sigma\boldsymbol E)(\boldsymbol y)|\left|K_{\left(\boldsymbol y,\widehat{\boldsymbol E}(\boldsymbol y)\right)}(\boldsymbol z, \boldsymbol\beta)\right|d\boldsymbol y.
    \end{aligned}
\end{equation}
By virtue of the estimates in the previous subsections and the expression \eqref{eq:J2}, we expect that the function $J$ can be used to indicate the conductive area $D$. When $\boldsymbol z$ lies outside the conductive region, the function $K_{(\boldsymbol y,\widehat{\boldsymbol E}(\boldsymbol y))}(\boldsymbol z, \boldsymbol\beta)$ consistently takes a small value, since $\boldsymbol z$ keeps a distance from all $\boldsymbol{y}\in D$, which shall give a small $J(\boldsymbol{z}, \boldsymbol{\beta})$ for all $\boldsymbol{\beta}\in S$. Furthermore, from Proposition \ref{prop:decay} we know that $J(\boldsymbol{z},\boldsymbol{\beta})$ uniformly tends to zero when $|\boldsymbol{z}|\to R$. These two properties ensure that the function $J$ always takes relatively small values outside the conductors. 

To better distinguish $D$ from $\Omega\setminus D$, we expect to choose appropriate $\boldsymbol\beta$s so that $J(\boldsymbol z,\boldsymbol\beta)$ takes relatively large values for $\boldsymbol z\in D$. A natural idea is to maximize $J(\boldsymbol z,\boldsymbol\beta)$, i.e., for every $\boldsymbol z\in B_R$ using the optimal $\boldsymbol\beta^*_{\boldsymbol z}$ such that
\begin{equation*}
    \boldsymbol\beta^*_{\boldsymbol z}=\arg\max_{\boldsymbol\beta\in S}J(\boldsymbol z,\boldsymbol\beta)=\arg\max_{\boldsymbol\beta\in S}\frac{\left|\langle\boldsymbol H^s(\cdot), \nabla_{\boldsymbol x}G(\cdot,\boldsymbol z)\times\boldsymbol\beta \rangle_\gamma\right|}{|\nabla_{\boldsymbol x}G(\cdot,\boldsymbol z)\times\boldsymbol\beta |_\gamma}.
\end{equation*}
However, this problem can be hard to solve. A simpler choice is to use
\begin{equation*}\label{eq:beta0}
    \boldsymbol\beta_{\boldsymbol z}=\arg\max_{\boldsymbol\beta\in S}\left|\langle\boldsymbol H^s(\cdot), \nabla_{\boldsymbol x}G(\cdot,\boldsymbol z)\times\boldsymbol\beta \rangle_\gamma\right|.
\end{equation*}
This $\boldsymbol\beta_{\boldsymbol z}$ can be calculated directly since
\begin{equation}\label{eq:beta}
    \begin{aligned}
    \boldsymbol\beta_{\boldsymbol z}=&\arg\max_{\boldsymbol\beta\in S}\left|\langle\boldsymbol H^s(\cdot), \nabla_{\boldsymbol x}G(\cdot,\boldsymbol z)\times\boldsymbol\beta \rangle_\gamma\right|\\
    =&\arg\max_{\boldsymbol\beta\in S}\left|\int_\Gamma \boldsymbol H^s(\boldsymbol x)\cdot\overline{(-\Delta_\Gamma)^\gamma\nabla_{\boldsymbol x}G(\boldsymbol{x},\boldsymbol z)\times\boldsymbol\beta}\,d\boldsymbol x\right|\\
    =&\arg\max_{\boldsymbol\beta\in S}\left|\,\overline{\boldsymbol \beta}\cdot\int_\Gamma\left(\boldsymbol H^s(\boldsymbol x)\times(-\Delta_\Gamma)^\gamma\nabla_{\boldsymbol x}G(\boldsymbol x,\boldsymbol z)\right)\,d\boldsymbol x\right|\\
    =&\left(\int_\Gamma\boldsymbol H^s(\boldsymbol x)\times(-\Delta_\Gamma)^\gamma\nabla_{\boldsymbol x}G(\boldsymbol x,\boldsymbol z)\,d\boldsymbol x\right)^\wedge,
    \end{aligned}
\end{equation}
where $(\cdot)^\wedge$ denotes vector normalization. The operator $(-\Delta_\Gamma)^\gamma$ acts on the variable $\boldsymbol x$. As we mentioned earlier, we have the analytical form of $(-\Delta_\Gamma)^\gamma \nabla_{\boldsymbol x}G(\boldsymbol{x},\boldsymbol{z})$, so that the computational complexity of the integral \eqref{eq:beta} is acceptable. 

Now we give the complete direct sampling method in Algorithm \ref{alg}. First, we solve the  homogeneous problem \eqref{eddy-E0} to obtain the background eddy current field $(\boldsymbol E_0,\boldsymbol H_0)$. With measurement $\boldsymbol{\mathcal M}=\boldsymbol H|_\Gamma$ we now get the magnetic perturbation field $\boldsymbol H^s=\boldsymbol{\mathcal M}-\boldsymbol H_0$ on the measurement surface $\Gamma$. Then for every possible $\boldsymbol z\in \Omega$ we calculate 
\begin{equation*}
    \boldsymbol \beta_{\boldsymbol z}=\left(\int_\Gamma
    \boldsymbol H^s(\boldsymbol x) \times(-\Delta_\Gamma)^\gamma \nabla_{\boldsymbol x}G(\boldsymbol x,\boldsymbol z)\,d\boldsymbol x\right)^\wedge, 
\end{equation*}
and the index function $I:\Omega\to[0,1]$ is defined as
\begin{equation}\label{eq:I}
    I(\boldsymbol z):=\cfrac{J(\boldsymbol z,\boldsymbol\beta_{\boldsymbol z})}{\max_{\boldsymbol{z}\in\Omega}J(\boldsymbol z,\boldsymbol\beta_{\boldsymbol z})}. 
\end{equation}
The conductive region $D$ is located where the index function \eqref{eq:I} attains larger values. 

\begin{algorithm}[!ht]
\caption{Direct Sampling Method for MIT}\label{alg}

\begin{algorithmic}[1]

\REQUIRE $\boldsymbol{\mathcal M}=\boldsymbol H|_\Gamma$, the magnetic field data on $\Gamma$.

\STATE Solve \eqref{eddy-E0} to get $(\boldsymbol E_0,\boldsymbol H_0)$,

\STATE Calculate perturbation field $\boldsymbol H^s=\boldsymbol{\mathcal M}-\boldsymbol H_0$ on data surface $\Gamma$,

\FOR{every possible $\boldsymbol z\in\Omega$}

\STATE Calculate $\boldsymbol \beta_{\boldsymbol z} =\left(\int_\Gamma\boldsymbol H^s(\boldsymbol x) \times(-\Delta_\Gamma)^\gamma \nabla_{\boldsymbol x}G(\boldsymbol x,\boldsymbol z)\,d\boldsymbol x\right)^\wedge$, 
    
\STATE Calculate $J(\boldsymbol z,\boldsymbol{\beta}_{\boldsymbol z})=\frac{|\langle\boldsymbol H^s(\cdot),\nabla_{\boldsymbol x}G(\cdot,\boldsymbol z)\times\boldsymbol \beta_{\boldsymbol z}\rangle_\gamma|}
{|\nabla_{\boldsymbol x}G(\cdot,\boldsymbol z)\times\boldsymbol \beta_{\boldsymbol z}|_\gamma}$,

\ENDFOR

\STATE Do normalization $I(\boldsymbol z)=\frac{J(\boldsymbol z,\boldsymbol{\beta}_{\boldsymbol z})}{\max_{\boldsymbol{z}\in\Omega}J(\boldsymbol z,\boldsymbol{\beta}_{\boldsymbol z})}$.

\end{algorithmic}

\end{algorithm}

We note that the integrals involved in computing $\boldsymbol{\beta}_{\boldsymbol{z}}$ and $J(\boldsymbol{z},\boldsymbol{\beta}_{\boldsymbol{z}})$ are evaluated using a numerical quadrature rule, since in practice the measurement $\boldsymbol{\mathcal M}$ is only available at discrete points on the surface $\Gamma$. We also note that  $\nabla_{\boldsymbol{x}} G(\boldsymbol{x,z})$ and $(-\Delta_\Gamma)^\gamma \nabla_{\boldsymbol{x}} G(\boldsymbol{x,z})$ can be precomputed for each sampling point $\boldsymbol{z}$, which helps to accelerate the imaging algorithm.

When multiple measurements $\{\boldsymbol{\mathcal M}_k\}_{k=1}^N$ are available, the direct sampling algorithm can be applied individually to each measurement $\boldsymbol{\mathcal{M}}_k$ to obtain the corresponding index function $I_k:\Omega\to[0,1]$, $k=1,\cdots,N$. Subsequently, these index functions are combined by computing their root-mean-square average
\begin{equation}
    \tilde I(\boldsymbol{z})=\frac{\sqrt{\frac1N\sum_{k=1}^N I_k^2(\boldsymbol{z})}}{\max_{\boldsymbol{z}\in\Omega}\sqrt{\frac1N\sum_{k=1}^N I_k^2(\boldsymbol{z})}}. \label{multipleI}
\end{equation}

\section{Numerical experiments}\label{sect4}

In this section, we present some numerical examples to demonstrate the efficiency of our direct sampling method in Algorithm \ref{alg} to solve the MIT problem. 

The sampling region is defined as the disk $\Omega=\{|\boldsymbol{x}|\le 1.0\}$, which has a constant magnetic permeability of $\mu=4\pi\times10^{-7}$. The angular frequency is set to $\omega=2\pi\times10^8$. The receivers are located on the spherical surface $\Gamma=\{|\boldsymbol{x}|=1.5\}$. In each example, the sampling domain contains several inhomogeneous inclusions, each with a constant conductivity of $\sigma=1.0$. To generate the observed data, we solve the forward problems \eqref{eddy-E} and \eqref{eddy-E0} using the finite element method, mainly by employing the auxiliary space Maxwell solver on the PHG(Parallel Hierarchical Grid) platform \cite{PHG2009}. 

For each example, we measure $N$ sets of data, all excited by the annular source current in a driving coil.  In our experiments, a total of $N=20$ coils were used to generate the data, with their centers uniformly distributed on a spherical surface of radius $r=1.5$. The coils are positioned at the vertices of a regular dodecahedron, as illustrated in Figure \ref{fig:coil_positions}. 

\begin{figure}[htbp]
\centering
\includegraphics[width=0.5\textwidth]{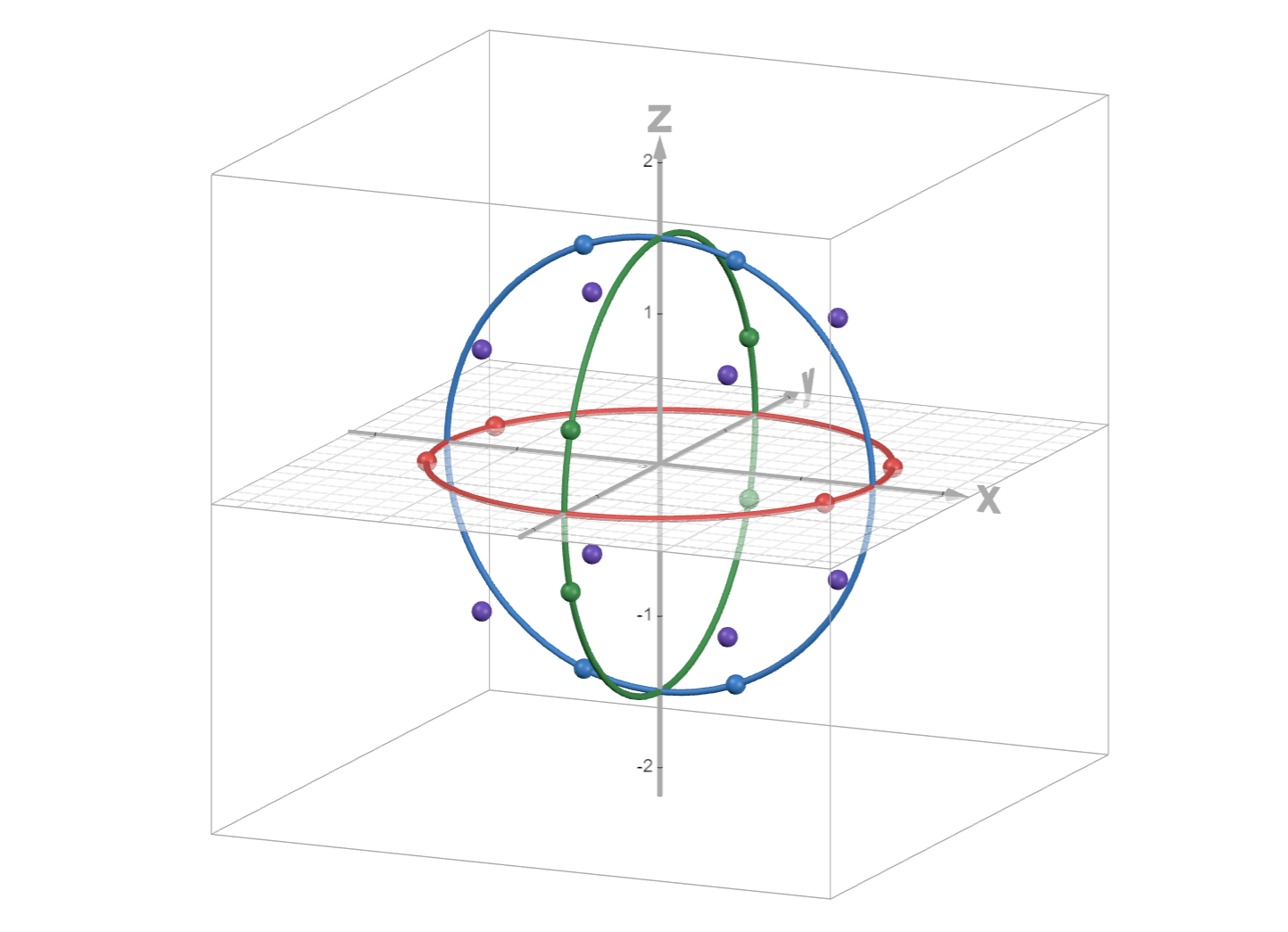}
\caption{Placement of the 20 coils at the vertices of a regular dodecahedron.}
\label{fig:coil_positions}
\end{figure}

These driving coils have identical shapes but are positioned in different orientations. Each driving coil can be regarded as a short and flat annular cylinder, with its central axis passing through the origin. It has an inner radius of $r_1=0.4$, an outer radius of $r_2=0.6$, and a height of $h=0.2$ (see Figure \ref{fig:geo}). Under the above constraints, the spatial position of a given coil is uniquely determined once the coordinates of its center are specified. 

During the $k$th measurement ($k=1,2,...,N$), we individually excite the $k$th coil and collect the corresponding magnetic field data $\boldsymbol{H}_k$ on the surface $\Gamma$. In our numerical experiments, we place $9812$ receivers on $\Gamma$. Therefore, each measurement data $\boldsymbol{\mathcal M}_k=\boldsymbol{H}_k|_{\Gamma}$ contains $9812$ three-dimensional vectors.


To specify the distribution $\boldsymbol{J}$ of the source current in the driving coil, we adopt a local coordinate system with its origin at the center of the coil $O'$. The local $x'$ axis is defined as the direction of $OO'$, and the perpendicular direction defines the $y'O'z'$ plane (see Figure \ref{fig:geo}). In this coordinate system, the source current $\boldsymbol{J}$ can be specified as
\begin{equation*}
    \boldsymbol{J}=\begin{cases}
        (0,-z',y'),&r_1\le\sqrt{y'^2+z'^2}\le r_2\quad\text{and}\quad|x'|\le h/2,\\
        (0,0,0),&\text{otherwise}.
    \end{cases}
\end{equation*}

\begin{figure}[htbp]
    \centering
    \begin{subfigure}{0.52\linewidth}
        \centering
        \includegraphics[width=\linewidth]{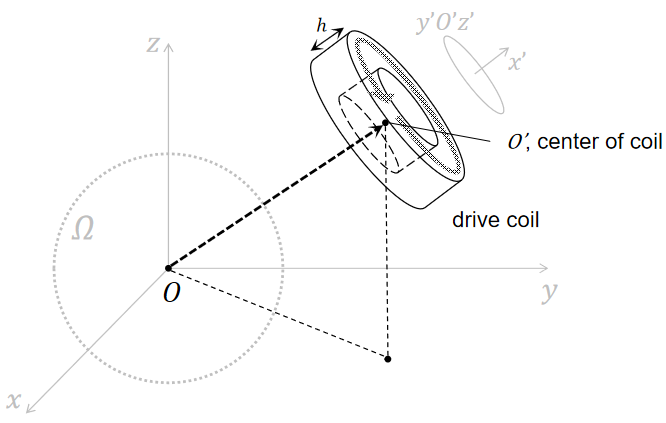}
        \caption{}
    \end{subfigure}
    \centering
    \begin{subfigure}{0.30\linewidth}
        \centering
        \includegraphics[width=0.9\linewidth]{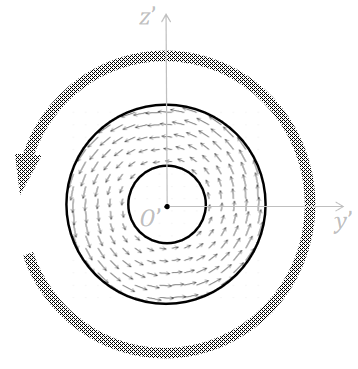}
        \caption{}
    \end{subfigure}
    \caption{Description of the geometry. Left: The driving coil and the local coordinate system $x'y'z'$. Right: Distribution of the current density in the driving coil under local coordinate system.}
    \label{fig:geo}
\end{figure}

Note that $\boldsymbol{J}$ is divergence-free, which satisfies our assumption. Let the source term $\boldsymbol{J}_0$ in the forward problem (2) be the current $\boldsymbol{J}_k$ generated by the $k$th driving coil, and then the solution $(\boldsymbol{E}_k,\boldsymbol{H}_k)$ shall give the corresponding data $\boldsymbol{\mathcal{M}}_k=\boldsymbol{H}_k|_{\Gamma}$, $k=1,\cdots,N$.

In addition, we introduce random noise to the data pointwisely in the form:
\begin{equation*}\label{eq:noise}
    \boldsymbol{\mathcal{M}}_k(\boldsymbol{x})=(\boldsymbol{1}+\epsilon\boldsymbol{\delta})\odot\boldsymbol{H}_k(\boldsymbol{x})=
    \begin{bmatrix}
        (1+\epsilon\delta_1)\cdot H_{k,1}(\boldsymbol{x})\\
        (1+\epsilon\delta_2)\cdot H_{k,2}(\boldsymbol{x})\\
        (1+\epsilon\delta_3)\cdot H_{k,3}(\boldsymbol{x})
    \end{bmatrix},
    \quad \boldsymbol{x}\in\Gamma.
\end{equation*}
Here, $\delta_1,\delta_2,\delta_3$ are complex random variables with their real and imaginary parts independently following a standard normal distribution, and $\epsilon$ is the relative noise level. In the following examples, we shall compare cases with $\epsilon=0$ and $\epsilon=20\%$. Our direct sampling method still performs well under this high noise level. This indicates that our algorithm is very robust in dealing with data noise.

During the experiments, we set $\gamma=4$ since the maximum of point spread functions $K$ will be sharp for large $\gamma$. We also perform post-processing on the reconstruction result \eqref{multipleI}. We find that the fourth power of the index function $\tilde I^4$ has the best presentation effect.

\paragraph{Example 1.} In this example, the conductive region $D$ is composed of two identical cubic inclusions, each having an edge length of $0.2$. These inclusions are located at $(0.40,0.41,0.0)$ and $(-0.40,-0.40,0.0)$ respectively. Figure \ref{fig:ex1} shows the reconstructed index functions $\tilde I$ and $\tilde I^4$ under different noise levels, and the conductor area is indicated in Figure \ref{fig:ex1} by the white box. We can see that the inclusions are well separated, and their locations are recovered pretty satisfactorily. This indicates that our direct sampling method has good separation capabilities when imaging multiple objects.

\begin{figure}[ht]
\centering
    \begin{subfigure}{0.32\linewidth}
        \centering
        \includegraphics[width=0.95\linewidth]{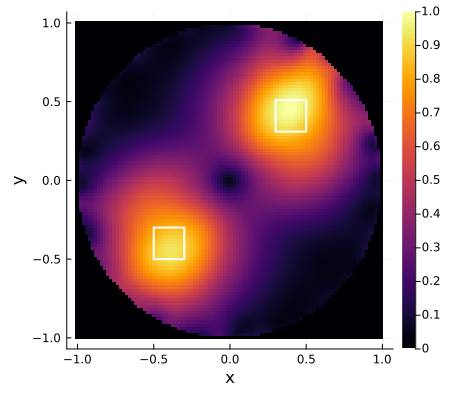}
    \end{subfigure}
    \centering
    \begin{subfigure}{0.32\linewidth}
        \centering
        \includegraphics[width=0.95\linewidth]{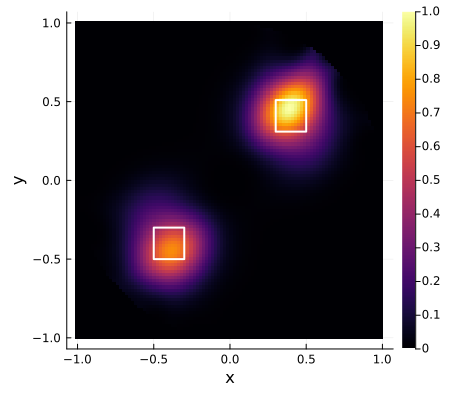}
    \end{subfigure}
    \centering
    \begin{subfigure}{0.32\linewidth}
        \centering
        \includegraphics[width=0.95\linewidth]{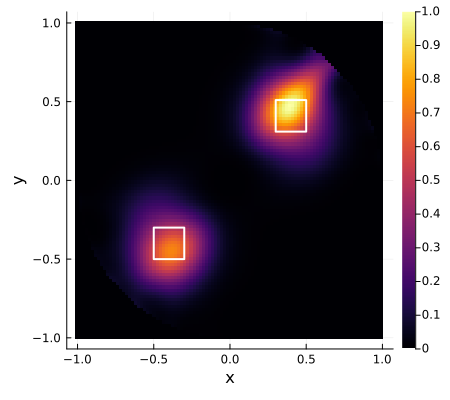}
    \end{subfigure}
\caption{Reconstructed medium image of example 1. Index function $\tilde I$ and $\tilde I^4$ on cross-section $\{x^2+y^2\le 1,\,z=0\}$. Left: $\tilde I$ with $\epsilon = 0$. Middle: $\tilde I^4$ with $\epsilon=0$. Right: $\tilde I^4$ with $\epsilon= 20\%$.  }
\label{fig:ex1}
\end{figure}

\paragraph{Example 2.} This example tests a medium with two cubic inclusions located on the same column, which are at the positions $(0.40,0.41,0.0)$ and $(0.40,-0.40,0.0)$, respectively. The size of the inclusions is the same as in example 1. The reconstructed images are shown in Figure \ref{fig:ex2}. From the figure, we observe that the scatterers are closer to each other and their reconstruction areas tend to merge. However, the locations of both inclusions are still recognizable. When the measurement data contains an $\epsilon=20\%$ relative noise, there is no significant change in our imaging results, which exhibits the good stability of our direct sampling method.

\begin{figure}[ht]
\centering
    \begin{subfigure}{0.32\linewidth}
        \centering
        \includegraphics[width=0.95\linewidth]{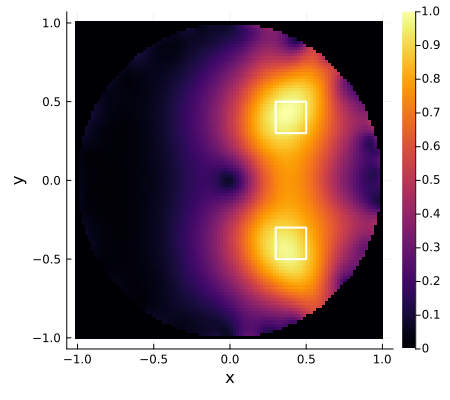}
    \end{subfigure}
    \centering
    \begin{subfigure}{0.32\linewidth}
        \centering
        \includegraphics[width=0.95\linewidth]{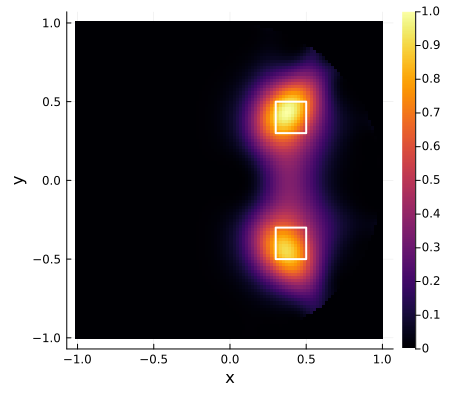}
    \end{subfigure}
    \centering
    \begin{subfigure}{0.32\linewidth}
        \centering
        \includegraphics[width=0.95\linewidth]{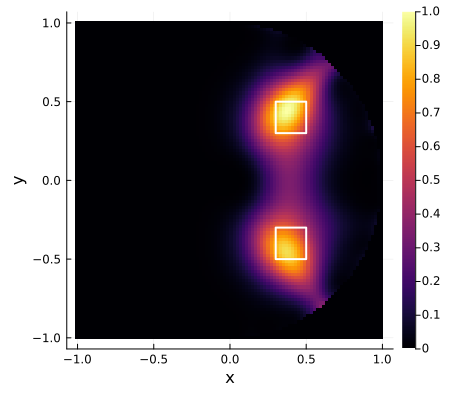}
    \end{subfigure}
\caption{Reconstructed medium image of example 1. Index function $\tilde I$ and $\tilde I^4$ on cross-section $\{x^2+y^2\le 1,\,z=0\}$. Left: $\tilde I$ with $\epsilon = 0$. Middle: $\tilde I^4$ with $\epsilon=0$. Right: $\tilde I^4$ with $\epsilon= 20\%$.  }
\label{fig:ex2}
\end{figure}

\paragraph{Example 3.} In this example, an L-shaped conductor is located on plane $\{z=0\}$ and has a thickness $0.2$. The specific position of the conductor is marked with the white frame lines in Figure \ref{fig:ex3}. The reconstruction results show that the imaging function is more pronounced at the L-shaped corner; overall, however, the index function correctly reflects the shape of the conductive area. The noise in the measurement data did not compromise the validity of the imaging results. Our algorithm remains effective and robust when imaging a single large conductor.

\begin{figure}[ht]
\centering
    \begin{subfigure}{0.38\linewidth}
        \centering
        \includegraphics[width=0.95\linewidth]{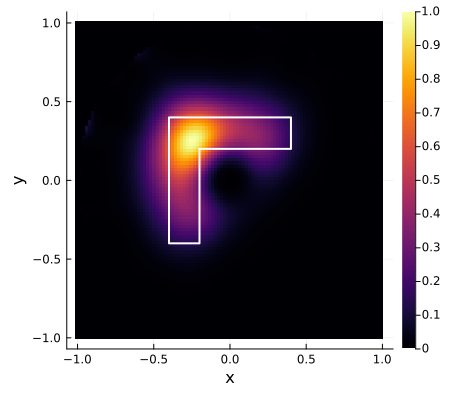}
    \end{subfigure}
    \centering
    \begin{subfigure}{0.38\linewidth}
        \centering
        \includegraphics[width=0.95\linewidth]{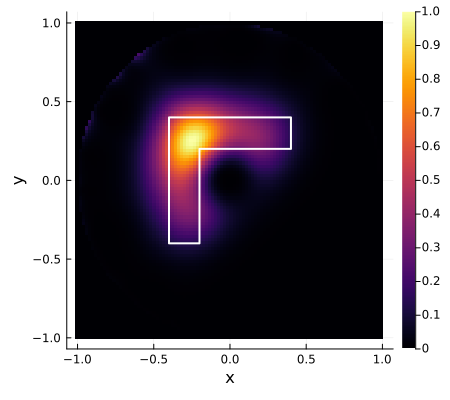}
    \end{subfigure}
\caption{Reconstructed medium image of example 3. Index function $\tilde I^4$ on cross-section $\{x^2+y^2\le1,\,z=0\}$. Left: $\tilde I^4$ with $\epsilon=0$. Right: $\tilde I^4$ with $\epsilon=20\%$. }
\label{fig:ex3}
\end{figure}

\paragraph{Example 4.} Four cubic conductors with an edge length of $0.12$ are set up in this example. Their centers are located at $(-0.3,-0.3,0.3)$, $(0.3,0.3,0.3)$, $(-0.3, 0.3, -0.3)$, $(0.3, -0.3, -0.3)$, respectively. In Figure $\ref{fig:ex4}$, we show the graph of index function $\tilde I^4$ on several cross-sections, including $\{x=-0.3\}$, $\{y=\pm0.3\}$ and $\{z=-0.3\}$. We observe that the index function $\tilde I^4$ could accurately reflect the position of the conductive area, even when the measurement data is subjected to a $20\%$ relative noise level. This demonstrates the effectiveness of our algorithm for three-dimensionally distributed conductors.

\begin{figure}[ht]
\centering
    \begin{subfigure}{0.32\linewidth}
        \centering
        \includegraphics[width=0.97\linewidth]{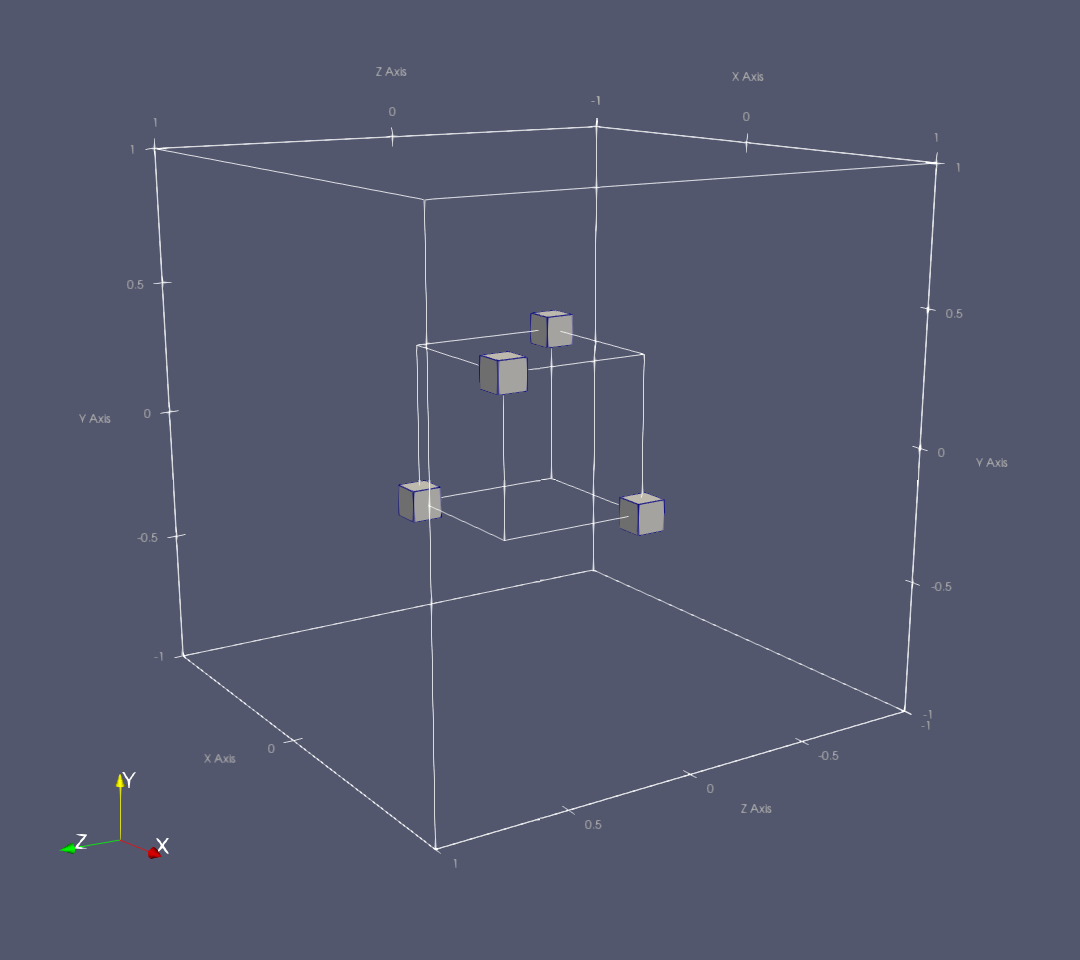}
    \end{subfigure}
    \centering
    \begin{subfigure}{0.32\linewidth}
        \centering
        \includegraphics[width=0.97\linewidth]{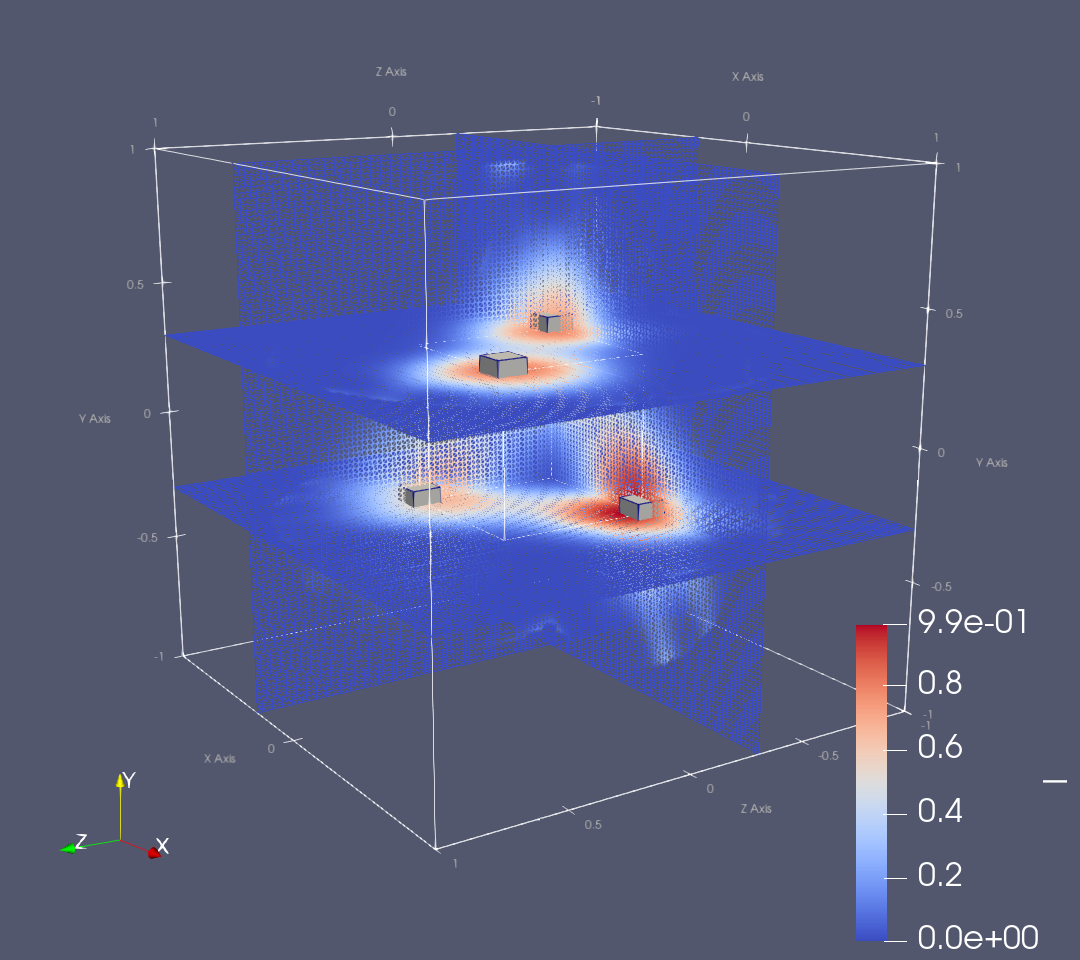}
    \end{subfigure}
    \centering
    \begin{subfigure}{0.32\linewidth}
        \centering
        \includegraphics[width=0.97\linewidth]{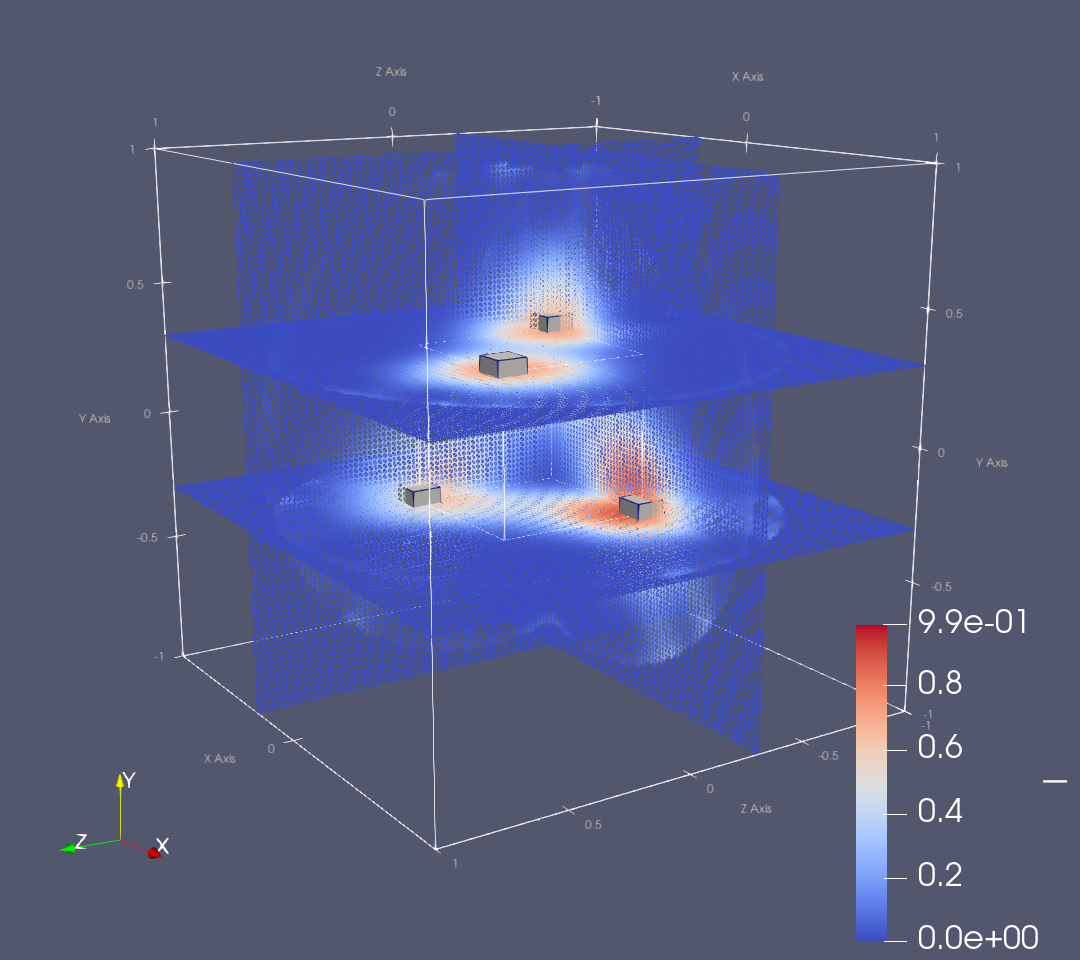}
    \end{subfigure}
\caption{Reconstructed medium image of example 4. Index function $\tilde I^4$ on several cross-sections. Left: Geometric position of conductors. Middle: $\tilde I^4$ with $\epsilon=0$. Right: $\tilde I^4$ with $\epsilon= 20\%$.  }
\label{fig:ex4}
\end{figure}

The above examples show that our direct sampling method for the MIT problem is capable of successfully identifying the conductive regions, and remains effective even when the measured data is contaminated with significant noise. Considering the severe ill-posedness of the MIT problem and the low computational cost of the reconstruction process, the numerical reconstruction results obtained by this newly proposed direct sampling method are quite satisfactory.

\section{Conclusion}\label{sect5}

This paper proposes a direct sampling method for solving magnetic induction tomography problems. The proposed imaging algorithm is characterized by its simplicity and non-iterative nature, where an index function is generated solely through the computation of vector inner products. In particular, the differentiation required during the computation utilizes analytical formulas rather than numerical differentiation, thereby significantly enhancing the computational speed and minimizing the numerical errors. In terms of theoretical analysis, we rigorously establish the decay property of the point spread functions $K$ near the boundaries, and further derive the explicit expressions for $K$ under specific conditions to visually demonstrate its decay behavior. This decay property ensures that the index function yields larger values inside the conductor and smaller values outside, thereby providing a mathematical justification for the imaging method. The efficacy of the algorithm is demonstrated through several numerical experiments. Our algorithm remains effective in reconstructing the conductive inclusions even with a noise level of $20\%$, demonstrating the robustness of the proposed algorithm. In summary, the method presented in this study provides a computationally efficient and theoretically sound approach to the magnetic induction tomography problem, indicating significant potential for practical applications in the field.

\bibliographystyle{abbrv}
\bibliography{references}
\end{document}